\documentclass[11pt]{amsart}
\usepackage{ryanmacros}
\input{ryanquivers.sty}

\usepackage[colorlinks=true, pdfstartview=FitV, linkcolor=blue, citecolor=blue, urlcolor=blue]{hyperref}

\usepackage{fullpage}
\CompileMatrices

\comment{
1. abstract: changed `graph' to `tree'

2. Prop 32a) and the proof is corrected by weakening it from a logical equivalence to an implication.  A partial converse follows from Prop 32b).

3. Defintion of `complexity': minor changes for consistency/clarity.

4. In Section 6., removed the hypothesis that the field $K$ be finite in several places, and updated the argument to support this

5. Section 2.1, added remark 3 on sheafy viewpoint of adjunction

6. Section 2.3 revised defintion of \latt{Q}{\sigma} to be 1 step.

7. Section 2.5 1 step def of rankM

References: added DOI to herschend, and 3 new references relating to referees suggestions in section 6
}

\begin{document}

\title{Rank Functions on Rooted Tree Quivers}
\author{Ryan Kinser}
\address{Department of Mathematics, University of Michigan, Ann Arbor, Michigan 48109}
\email{kinser@umich.edu}
\thanks{This material was based upon work supported under a National Science Foundation Graduate Research Fellowship, and NSF Grant DMS 0349019}

\begin{abstract}
The free abelian group $R(Q)$ on the set of indecomposable representations of a quiver $Q$, over a field $K$, has a ring structure where the multiplication is given by the tensor product.  We show that if $Q$ is a rooted tree (an oriented tree with a unique sink), then the ring $R(Q)_{red}$ is a finitely generated $\Z$-module (here $R(Q)_{red}$ is the ring $R(Q)$ modulo the ideal of all nilpotent elements).  We will describe the ring $R(Q)_{red}$ explicitly, by studying functors from the category $\repq$ of representations of $Q$ over $K$ to the category of finite dimensional $K$-vector spaces.
\end{abstract}
\maketitle

%%%%%%%%%%%%%%%%%%%%%%%%%%%%%%%%%%%%%%%%%%%%%%%
%						 			INTRODUCTION						%
%%%%%%%%%%%%%%%%%%%%%%%%%%%%%%%%%%%%%%%%%%%%%%%
\section{Introduction}

A \textbf{quiver} is just a directed graph $Q=(Q\verts, Q\arrows, t, h)$, where $Q\verts$ is a vertex set, $Q\arrows$ is an arrow set, and $t, h$ are functions from $\qarrows$ to $\qverts$ giving the tail and head of an arrow, respectively.  We assume $Q\verts$ and $Q\arrows$ are finite in this paper.
For any quiver $Q$ and field $K$, there is a category $\rep_K (Q)$ of representations of $Q$ over $K$.  An object $V$ of $\rep_K (Q)$ is an assignment of a finite dimensional $K$-vector space $V_x$ to each vertex $x \in Q\verts$, and an assignment of a $K$-linear map $V_a \colon V_{ta} \to V_{ha}$ to each arrow $a \in Q\arrows$.  For any path $p$ in $Q$, we get a $K$-linear map $V_p$ by composition.  Morphisms in $\rep_K (Q)$ are given by linear maps at each vertex which form commutative diagrams over each arrow; see the book of Assem, Simson, and Skowro\'nski \cite{assemetal} for a precise definition of morphisms, and other fundamentals of quiver representations.  We will fix some arbitrary field $K$ throughout the paper and hence omit it from notation when possible.

There is also a natural tensor product of quiver representations, induced by the tensor product in the category of vector spaces (cf. \cite{strassenasymptotic, herschtensorjpaa}).  Concretely, it is the ``pointwise'' tensor product of representations defined by
\[
(V \otimes W)_x := V_x \otimes W_x
\]
for each vertex $x$, and similarly for arrows.  This tensor product gives the category $\repq$ the structure of a \textbf{tensor category} in the sense of \cite{DMtannakian}, and, along with direct sum, endows the set of isomorphism classes in $\repq$ with a semiring structure.  The associated ring $R(Q)$ is the \textbf{representation ring} of $Q$ (cf. \S\ref{repringsect}), which is commutative with identity $\id_Q$, where we define $(\id_Q)_x := K$ and $(\id_Q)_a := id$ for all vertices $x$ and arrows $a$.  

For a quiver $Q$ which is not of Dynkin or extended Dynkin type, the problem of classifying its indecomposable representations is unsolved and very difficult, to say the least.  Such a quiver is said to be of ``wild representation type'' (cf. \cite{drozdtamewild}) and has families of indecomposable representations depending on arbitrarily large numbers of parameters.
This limits the effectiveness of an enumerative approach to studying tensor products of quiver representations (as opposed to, say, tensor products of representations of finite groups or the classical groups).  Alternatively, we seek to describe $R(Q)$ in abstract terms, and translate properties of $R(Q)$ into properties of the tensor product in $\repq$.
For example, the main result of this paper has two equivalent formulations, the first of which (Theorem \ref{mainthm}) can be stated in a simplified form here:

\begin{theoremnonum}
When $Q$ is a rooted tree quiver, $R(Q)_{red}$ is generated as a $\Z$-module by a finite set of explicit representations of $Q$.
(Here $A_{red}$ is the reduction of a ring $A$, that is, $A$ modulo its ideal of nilpotent elements.)
\end{theoremnonum}

This theorem has an equivalent formulation (Theorem \ref{splittingthm}) as a splitting principle for large tensor powers $V^{\otimes n}$ of a fixed representation $V$, which makes no mention of the representation ring.

Our main tools for studying $R(Q)$ are global tensor functors which can be used to construct various homomorphisms from $R(Q)$ to other rings.  We present a theorem summarizing the properties of global tensor functors.  Here $\vscat$ denotes the category of finite dimensional vector spaces over a field $K$.

\begin{theorem}[\cite{kinserrank}]\label{kinserrankthm}
Let $Q$ by any connected quiver.  There is a \textbf{global tensor functor}
\[
\rkf_Q\colon \repq \to \repq
\]
with the following properties:
\begin{enumerate}[(a)]
\item The functor $\rkf_Q$ commutes with direct sum, tensor product, Schur functors (when $\charac K = 0$), and duality; we also have that $\rkf_Q (\id_Q) = \id_Q$.
\item Let $V \in \repq$.  Then for every arrow $a$ of $Q$, the linear map $(\rkf_Q (V))_a$ is an isomorphism. Hence the isomorphism class of the functor 
\[
\rankfunc_Q V:=(\rkf_Q (V))_x \colon \repq \to \vscat
\]
is independent of the vertex $x$;  we call this functor the \textbf{global rank functor} of $Q$.
\item When $Q$ is a tree and $V \in \repq$, the representation $\rkf_Q (V)$ is isomorphic to a direct summand of $V$.  More precisely, for any indecomposable representation $W$ of $Q$,
\[
\rkf_Q (W) =
\begin{cases}
W & \text{if }W \simeq \id_Q \\
0 & \text{if }W \not \simeq \id_Q \\
\end{cases}.
\]
\end{enumerate}
\end{theorem}

Combining global rank functors with pullback along maps of directed graphs (cf. \S \ref{rankfunctorsect}), it is possible to construct more non-zero functors from $\repq$ to the category of finite dimensional $K$-vector spaces  which respect direct sum and tensor product.  We call such a functor a \textbf{rank functor} on $Q$.  A rank functor $F$ induces a ring homomorphism $f \colon R(Q) \to \Z$, defined on representations $V$ by $f(V) = \dim_K F(V)$ and extended by linearity to $R(Q)$. These functions are called \textbf{rank functions} of $Q$.

A \textbf{rooted tree quiver} is a directed graph $Q$, whose underlying graph is a tree, and which has a unique sink $\sigma$ called the \textbf{root} of $Q$. We sometimes write $(Q, \sigma)$ if we want to specify the root.  Equivalently, one may give a graph which is a tree and specify a root vertex, with the convention that all edges are oriented towards this root.
Rooted trees have the convenient property that all of their connected subquivers are rooted trees, thus lending themselves to inductive proof methods.
The path algebra of a rooted tree quiver is hereditary and right serial; conversely, the ordinary quiver of any basic, hereditary, right serial $K$-algebra is a rooted tree (cf. \cite[Thm.~2.6]{assemetal}).  It should also be noted that the main results of this paper hold (with minor changes in terminology) for a quiver $Q$ which is a tree with a unique \emph{source}. In this case, $Q^{op}$ (the quiver obtained by switching the heads and tails of all arrows) is a tree with a unique sink, and the standard duality between representations of $Q$ and representations of $Q^{op}$ induces a ring isomorphism $R(Q) \cong R(Q^{op})$.
Global rank functors commute with duality also, so the methods used for the unique sink case can be applied in a straightforward way to treat the unique source case.

The paper is organized as follows.
Section \ref{rankfunctorsect} establishes basic tools for studying representations of a rooted tree quiver $Q$ via quivers over $Q$.  This leads to the construction of various distinct rank functors on $Q$ within a combinatorial framework.
The focus is shifted from rank functors to representations of $Q$ in Section \ref{reducedrepssect}, by constructing a set of ``reduced'' representations that are in some sense dual to the rank functors of the previous section.  Then the combinatorics of these rank functors can be utilized to obtain information about tensor products of reduced representations, and morphisms between them.
In Section \ref{repringsect}, we make use of the properties of reduced representations to study representation rings.  First, we give an algebraic framework for reducing questions about tensor products of quiver representations on any quiver (not just rooted trees) to tensor products of representations with full support.  The technical tool developed is a decomposition of $R(Q)$ into a direct product of rings, with one factor for each connected subquiver of $Q$, such that a representation $V$ has nonzero image in the factor corresponding to $P \subseteq Q$ if and only if the support of $V$ contains $P$.
Then by introducing a property of representations of rooted tree quivers which generalizes the support of a representation, we refine this direct product decomposition of $R(Q)$.

The two theorems mentioned above on the representation rings of rooted tree quivers are stated and proven in Section \ref{mainresultsect}.  These comprise the main results of this paper.
First, we show the equivalence of the two theorems, then prove the result by induction on the complexity of a rooted tree.
There are two cases in the proof of the main theorem: one is essentially combinatorial, using the representation ring form of the result as the induction hypothesis; the other is essentially computational, using the splitting principle form of the result as the induction hypothesis.
In Section \ref{conclusionsect}, we introduce the name \textbf{finite multiplicative type} for a quiver whose reduced representation ring is module finite over $\Z$.  Having given a large class of such quivers (the main result), it is then natural to try to classify all of them.  To this end, we show that the class of quivers of finite multiplicative type is minor closed, and can only include trees; but we also give an example of a tree quiver (of tame representation type even) which is not of finite multiplicative type.  By an application of Kruskal's Tree Theorem, this property can be characterized by a \emph{finite} set of forbidden minors.  

\subsection*{Acknowledgements}
I would like to thank my advisor Harm Derksen for conjecturing the main result, and for many discussions leading to its proof.  I would also like to thank Hugh Thomas for comments and corrections to the first version of the paper, and John Stembridge for helpful discussion on minor closed sets of graphs.  Several other insightful comments and corrections were provided by an anonymous referee.

%%%%%%%%%%%%%%%%%%%%%%%%%%%%%%%%%%%%%%%%%%%%%%%
%				 RANK FUNCTORS	FOR ROOTED TREES				%
%%%%%%%%%%%%%%%%%%%%%%%%%%%%%%%%%%%%%%%%%%%%%%%
\section{Rank Functors on Rooted Trees}
\label{rankfunctorsect}
%%%%%%%%%%%%%%%%%%%%%%%%%%%%%%%%%%%%%%%%%%%%%%%
%						PUSHFORWARD AND PULLBACK						%
%%%%%%%%%%%%%%%%%%%%%%%%%%%%%%%%%%%%%%%%%%%%%%%
\subsection{Pushforward and pullback of representations}

Maps of directed graphs and covering quivers have been used in works such as \cite{riedtmanntranslationquivers, gabrieluniversalcover, bongartzgabrielcoverings} to study representations of quivers, or more generally, of finite dimensional algebras.
In this paper, we will not be interested in maps that are topological coverings of some base quiver, but rather maps that encode combinatorial data about the base quiver.
By definition, a map of directed graphs $f \colon Q' \to Q$ sends vertices to vertices and arrows to arrows, and satisfies $tf(a) = f(ta)$ and $hf(a) = f(ha)$ 
for each arrow $a \in Q'\arrows$.
A \textbf{quiver over $Q$} is a pair $(Q', f)$ where $Q'$ is a quiver, and $f\colon Q' \to Q$ a map of directed graphs called the structure map of $(Q', f)$.
To simplify the notation, we consider the maps $V_a$ of a representation $V$ to be defined on the total vector space $\bigoplus_{x \in Q\verts} V_x$ by taking $V_a (v) = 0$ for $v \in V_y$, when $y \neq ta$.
The \textbf{pullback} $f^* W \in \rep(Q')$ of a representation $W \in \repq$ along a map of directed graphs $f\colon Q' \to Q$ is given by
\[
(f^*W)_x := W_{f(x)} \qquad x \in Q'\verts \qquad \qquad (f^*W)_a := W_{f(a)} \qquad a \in Q'\arrows,
\]
and the \textbf{pushforward} $f_* V \in \repq$ of a representation $V \in \rep(Q')$ is given by
\[
(f_* V)_x := \bigoplus_{y \in f^{-1}(x)} V_y \qquad x \in Q\verts \qquad \qquad (f_* V)_a := \sum_{b \in f^{-1}(a)} V_b \qquad a \in Q\arrows .
\]
A map of rooted tree quivers $f\colon (Q', \sigma') \to (Q, \sigma)$ is said to be \textbf{root preserving} if $f(\sigma') = \sigma$.

There is a categorical view of the above definitions. Write $\mathscr{Q}$ for the free category on the directed graph $Q$: that is, the objects of $\mathscr{Q}$ are the vertices of $Q$, and the morphisms are paths in $Q$.  A representation of $Q$ is the same thing as a functor
\[
V\colon \mathscr{Q} \to \vscat,
\]
and a map of directed graphs $f \colon Q' \to Q$ induces a functor $F \colon \mathscr{Q}' \to \mathscr{Q}$.  Then pullback is just the composition of functors
\[
f^*V \colon \mathscr{Q}' \xto{F} \mathscr{Q} \xto{V} \vscat,
\]
and now we will show that pushforward is its left adjoint.

%%%%%%%%%%%	ADJOINT FUNCTORS	%%%%%%%%%%%%%%
\begin{prop}\label{adjointprop}
Let $f\colon (Q', \sigma') \to (Q, \sigma)$ be a root preserving map between rooted tree quivers.
Then the pushforward functor $f_*$ is left adjoint to the pullback functor $f^*$.
\end{prop}
\begin{proof}
For $X, Y \in \repq$, the map of vector spaces
\begin{equation*}
\begin{split}
\bigoplus_{z \in Q\verts}	\Hom_K (X_z, Y_z) 	 &\xto{c^Y_X} \bigoplus_{a \in Q\arrows}  \Hom_K (X_{ta}, Y_{ha})  \\
		 (\phi_z)_{z \in Q\verts} 			 &\mapsto		(Y_a \phi_{ta} - \phi_{ha} X_a)_{a \in Q\arrows}
\end{split}
\end{equation*}
has kernel precisely $\Hom_{Q}(X, Y)$, by the definition of morphisms in $\repq$, and similarly for $Q'$.  We use this to give a natural isomorphism $\Hom_Q (f_* V, W) \cong \Hom_{Q'} (V, f^*W)$.

For $V \in \rep(Q')$ and $W \in \repq$, there is a natural isomorphism of vector spaces
\begin{equation}\label{adjointeqn}
\begin{split}
\bigoplus_{x \in Q\verts} \Hom_K (V_x, (f^* W)_x) &\cong \bigoplus_{y \in Q'\verts} \Hom_K ((f_*V)_y, W_y) \\
(\phi_x)_{x \in Q\verts} &\mapsto \left(\sum_{x \in f^{-1}(y)} \phi_x \right)_{y \in Q'\verts}.
\end{split}
\end{equation}
Then we can compute
\[
c^{f^*W}_{V}(\phi_x) = \sum_{b \in {Q\arrows}} \sum_{c \in f^{-1}(b)} W_b \phi_{tc} - \sum_{b \in {Q\arrows}} \sum_{c \in f^{-1}(b)} \phi_{hc} V_c= \sum_{c \in Q'\arrows} W_{f(c)} \phi_{tc} - \sum_{c \in Q'\arrows} \phi_{hc} V_c
\]
and using the identification (\ref{adjointeqn}) we also compute
\[
c^W_{f_*V}(\phi_x) = \sum_{b \in {Q\arrows}} \sum_{x \in f^{-1}(tb)} W_b \phi_x - \sum_{b \in {Q\arrows}} \sum_{y \in f^{-1}(hb)} \sum_{\substack{c \in f^{-1}(b) \\ hc = y}} \phi_y V_c= \sum_{\substack{(x, b)\in Q'\verts \times Q\arrows \\ x = f^{-1}(tb)}} W_b \phi_x - \sum_{c \in Q'\arrows} \phi_{hc} V_c.
\]

Now for each pair $(x, b) \in Q'\verts \times Q\arrows$ such that $f(x) = tb$, there exists a unique arrow $c \in Q'\arrows$ such that $f(c) = b$ and $tc = x$.  This is because, in a rooted tree quiver, every vertex except the sink has a unique outgoing arrow, and the assumption that $f(\sigma') = \sigma$ guarantees that $x \neq \sigma'$.  This allows us to identify the terms $W_{f(c)} \phi_{tc}$ with the terms $W_b \phi_x$ in the sums above.
Thus, the isomorphism (\ref{adjointeqn}) induces a natural isomorphism $\im c^W_{f_*V} \cong \im c^{f^*W}_{V}$, and so the kernels of $c^{f^*W}_{V}$ and $c^W_{f_*V}$ are naturally isomorphic also.
\end{proof}

%\begin{remark}
We sketch a more elegant proof of the preceding proposition using sheaves, which was provided by an anonymous referee.  A rooted tree quiver $Q$ can be viewed as a poset by declaring $x \geq y$ for any two vertices $x,y$ such that there exists a path from $x$ to $y$.  Then we can define a topology on $Q\verts$ by taking the open sets to be the dual order ideals (also known as upper sets) of $Q$ \cite[p.~100]{stanleyenumcombin}.  The category of sheaves of finite dimensional vector spaces on this topological space is equivalent to the category $\rep (Q^{op})$, where $Q^{op}$ denotes the quiver $Q$ with the orientations of the arrows reversed.  Stalks of sheaf correspond to the vector spaces associated to the vertices in a representation.

A continuous map $f \colon Q' \to Q$ under this topology on rooted tree quivers is just a map of directed graphs, and when $f$ is root preserving we find that pushforward and pullback of sheaves agree with the corresponding notions for quiver representations under the equivalence.  So the adjointness of $f^*$ and $f_*$ for sheaves implies that
$
f^* \colon \rep(Q^{op}) \to \rep(Q'^{op})
$
is left adjoint to $f_* \colon  \rep(Q'^{op}) \to \rep(Q^{op})$, and dualizing we get the adjoint pair $(f_*, f^*)$ of the proposition.
%\end{remark}

%Need root preserving: for example, the inclusion of the tail vertex into $A_2$.

\subsection{Rooted tree quivers}\label{rootedtreesect}
Let $(Q, \sigma)$ be a rooted tree quiver.  Then any $Q' \xto{f} Q$ induces a rank functor on $Q$ by composing pullback along $f$ with the global rank functor of $Q'$. This gives the possibility of constructing infinitely many rank functors on $Q$, \emph{a priori}, but it is possible for distinct quivers over $Q$ to give isomorphic rank functors.  It turns out that when $Q$ is a rooted tree, there is a finite set of  ``reduced'' quivers over $Q$ which give all rank functors on $Q$ that can be obtained in the way just described.  Furthermore, there is a natural partial ordering on this set of rank functors.

The technique of many proofs in this paper is induction on the number of vertices of $Q$, using the observation that every connected subquiver of a rooted tree is again a rooted tree.  A \textbf{subquiver} $P$ of a quiver $Q$ is given by a subset of vertices $P\verts \subseteq Q\verts$ and a subset of arrows $P\arrows \subseteq Q\arrows$ with the same orientation as in $Q$.  We will usually assume that subquivers are connected, so that the global rank functor of a subquiver is defined.
We can build any rooted tree quiver by two fundamental processes, which we call ``extension'' and ``gluing''.  

We say that a rooted tree quiver $(Q, \sigma)$ is obtained from a subquiver $(P, \tau) \subset Q$ by \textbf{extension} if
\[
Q\verts = P\verts \cup \{ \sigma \} \qquad \text{and} \qquad Q\arrows = P\arrows \cup \{ \alpha \}
\]
where $t\alpha = \tau$ and $h\alpha = \sigma$.
\comment{ Diagrammatically, this looks like
\[
Q= \qquad \extensiondiagram
\]}
%(***picture)
Note that if such a $P$ exists (in a rooted tree quiver), it is unique.

In any rooted tree quiver which is not an extension of a subquiver, in the above sense, there exists a unique maximal collection of subquivers $\{Q_i \subsetneq Q\}_{i \in I}$ such that $Q$ is obtained by \textbf{gluing} the $Q_i$ at their sinks:  that is, $Q = \bigcup_i Q_i$, and $Q_i \cap Q_j = \sigma$ for $i \neq j$.  In this case, we write
\[
Q = {\coprod_{i \in I}}^\sigma Q_i.
\]
Each of these $Q_i$ is an extension of a unique subquiver $P_i \subsetneq Q_i$.
The notion of $Q$ being glued from any two subquivers $P, S \subset Q$ at their sinks is similarly defined.
%(***picture)

Many proofs in this paper will use induction on the number of vertices of $(Q, \sigma)$.  If there is a unique arrow in $Q$ whose head is $\sigma$, then $Q$ is obtained from a quiver with fewer vertices by extension; if there is more than one arrow with head $\sigma$, then $Q$ is obtained by gluing two quivers $P$ and $S$ at $\sigma$, each with fewer vertices than $Q$.  Hence, to recursively define a construction depending on $Q$, or to prove any property of $Q$ by induction, we just need to start with the rooted tree quiver with one vertex, and then show how to proceed for quivers obtained by extension or gluing.
The base case is usually trivial and will be omitted.  In the gluing case, \emph{a priori} a definition or property could depend on the choice of $P$ and $S$ in $Q$.  However, this will not be the case for any properties or definitions of this paper, which can be easily seen in any particular case by working with the unique maximal collection of $Q_i$'s as in the definition of gluing above.  So, for simplicity, in the gluing case we will always use the setup of two subquivers $P$ and $S$ with it implicit that the result does not depend on the choice of $P$ and $S$ used.
In the course of such proofs, the notation above will be assumed to be in place unless explicitly stated otherwise.

Our first task will be to establish a connection between the global rank functor of a rooted tree quiver, and the global rank functors of its subquivers.
The use of two functors $\surjrep_Q$ and $\injrep_Q$ will only be needed for a few technical lemmas, so their properties will not be reviewed in depth; the interested reader may see the paper \cite{kinserrank} for details.
Briefly, for $V \in \repq$, the functor
\[
\surjrep_Q \colon \repq \to \repq
\]
gives the unique maximal epimorphic subrepresentation $\surjrep_Q (V) \subseteq V$, where a representation $W \in \repq$ is said to be \textbf{epimorphic} when $W_a$ is surjective for every arrow $a \in Q\arrows$.  Dually, $\injrep_Q$ gives the unique maximal quotient of $V$ such that the maps at all arrows are injective.  Then $\rkf_Q$ is defined as the image functor of the composition $\surjrep_Q \into id_{\repq} \onto \injrep_Q$.  

\begin{lemma}\label{rootedtreeranklemma}
If $(Q, \sigma)$ is a rooted tree quiver, then $(\injrep_Q V)_\sigma = V_\sigma$, and hence
\[
\rank_Q V = (\surjrep_Q V)_\sigma \subseteq V_\sigma .
\]
\end{lemma}
\begin{proof}
For $x \in Q\verts$, denote by $V_{x \to \sigma}$ the linear map $V_p$ for $p$ the unique path from $x$ to $\sigma$.  Then
\[
W_x := \frac{V_x}{\ker V_{x \to \sigma}}
\]
defines a monomorphic quotient representation of $V$.  But any intermediate quotient $V \onto W' \onto W$ is not monomorphic, because by definition of $W$ there must be some path $p$ such that $W'_p$ has non-trivial kernel.  Hence $\injrep_Q V = W$, and since $V_{\sigma \to \sigma}$ is the identity map, we get the first statement.  Then the second statement follows from the definition of $\rank_Q V$.
\end{proof}

If $P \subset Q$ is a subquiver, and $V \in \repq$, denote by $V|_P$ the restriction of $V$ to $P$.  The previous lemma can be used to inductively construct global rank functors.

\begin{lemma}\label{glueextranklemma}
If $(Q, \sigma)$ is obtained by extension from $(P, \tau)$ via an arrow $\alpha$, then the global rank functor of $Q$ can be calculated as
\[
\rank_Q V = V_\alpha (\rank_P (V|_P))
\]
where we consider $\rank_P (V|_P) \subseteq V_\tau$.
If $Q$ is obtained by gluing $P$ and $S$, then we have
\[
\rank_Q V = \rank_P ( V|_P) \cap \rank_S (V|_S)
\]
where the intersection is taken in $V_\sigma = (V|_P)_\sigma = (V|_S)_\sigma$.
\end{lemma}
\begin{proof}
\extcase Define an epimorphic subrepresentation $E \subseteq V$ by
\[
E_x =
\begin{cases}
\surjrep_P ( V|_P)_x & x \in P\verts \\
V_\alpha (E_\tau) & x = \sigma \\
\end{cases}.
\]
If $M \subseteq V$ is any epimorphic subrepresentation, then we have an inclusion $M|_P \subseteq \surjrep_P ( V|_P) = E|_P$ by the universal property of $\surjrep_P$, and the epimorphic property gives
\[
M_\sigma \subseteq V_\alpha (M_\tau) \subseteq V_\alpha (E_\tau ) = E_\sigma .
\]
Hence any epimorphic subrepresentation $M \subseteq V$ is contained in $E$, so we have $E = \surjrep_Q V$ by the universal property of $\surjrep_Q$.  Using Lemma \ref{rootedtreeranklemma}, we get
\[
\rank_Q V = (\surjrep_Q V)_\sigma = E_\sigma = V_\alpha (E_\tau) = V_\alpha ( \surjrep_P ( V|_P)_\tau ) = V_\alpha (\rank_P ( V|_P)) .
\]

\gluecase We want to show that $\rank_Q V= Z$, where
\[
Z := \rank_P ( V|_P) \cap \rank_S (V|_S).
\]
By Lemma \ref{rootedtreeranklemma}, we have $\rank_Q V = (\surjrep_Q V)_\sigma$ and $Z =  \surjrep_P (V|_P)_\sigma \cap \surjrep_S (V|_S)_\sigma $.

$\supseteq:$ First we construct an epimorphic subrepresentation $M \subseteq V$ such that $M_\sigma = Z$.  Retaining the notation $V_{x \to \sigma}$ from the previous lemma, define $M$ by 
\[
M_x :=
\begin{cases}
(V_{x \to \sigma})^{-1}(Z) \cap \surjrep_P (V|_P)_x		& x \in P\verts \\
(V_{x \to \sigma})^{-1}(Z) \cap \surjrep_S (V|_S)_x		& x \in S\verts
\end{cases}.
\]
It is straightforward to check that $M$ is an epimorphic subrepresentation of $V$, so $M$ is contained in $\surjrep_Q V$.  Hence $Z = M_\sigma \subseteq (\surjrep_Q V)_\sigma = \rank_Q V$.

$\subseteq:$  The universal properties of $\surjrep_P$ and $\surjrep_S$ give
\[
(\surjrep_Q V)|_P \subseteq \surjrep_P (V|_P) \qquad \text{and} \qquad (\surjrep_Q V)|_S \subseteq \surjrep_S (V|_S)
\]
so again using Lemma \ref{rootedtreeranklemma} we have
\[
\rank_Q V = (\surjrep_Q V)_\sigma \subseteq Z . \qedhere
\]
\end{proof}

Our goal is to construct as many distinct rank functors on $Q$ as possible.
We motivate the general procedure with two examples.

\begin{example}\label{exampleone}
The three subspace quiver $Q$ can be obtained from three $A_2$ quivers 
%\begin{figure}[h]
$$Q_1=
\vcenter{\xymatrix@=12pt{
		& \sigma	\\
1 \ar[ur]	& 	}}
\qquad
Q_2=
\vcenter{\xymatrix@=12pt{
 \sigma	\\
2 \ar[u]	}}
\qquad
Q_3 = \vcenter{\xymatrix@=12pt{
 \sigma	&	\\
		& 3 \ar[ul]	}}
$$
%\caption{}
%\end{figure}
by gluing them at their sinks.
\[
Q= \threesubspace{\sigma}{1}{2}{3}
\]
On each $Q_i$, the global rank functor is just given on a representation $V = V_i \xto{A_i} V_\sigma$ by
\[
\rank_{Q_i}(V) =\im A_i \subseteq V_\sigma ,
\]
and $F_i (V) = V_\sigma$ is another rank functor on $Q_i$.

To any subset $J \subseteq \{1, 2, 3 \}$ and $V \in \repq$, we can use the rank functors on the subquivers $Q_i$ to get a vector space
\[
\rank_J 
%\left(\threesubspacemaps{\sigma}{1}{2}{3}{A_1}{A_2}{A_3} \right)
(V) = \bigcap_{j \in J} \rank_{Q_j} (V|_{Q_j}) \subseteq V_\sigma
\]
and in fact this defines a rank functor on $Q$.  This gives an ordering reversing correspondence between the lattice of subsets of $B_3 = \{1, 2, 3 \}$, and a set of rank functors on $Q$, with the latter ordered by inclusion of functors.  But not only are these rank functors built from rank functors on smaller quivers, the partial order on them also comes from these smaller quivers, in the following sense.
Let $\bold{2} = \{\hat{0}, \hat{1}\}$ be ordered by $\hat{0} < \hat{1}$.  
If, for each $i$, we associate $\hat{0}$ with $F_i$ and $\hat{1}$ with $\rank_{Q_i}$, then the isomorphism of posets $B_3 \simeq \bold{2} \times \bold{2} \times \bold{2}$ induces the ordering on rank functors of $Q$ by associating a product of elements on the right hand side with intersection of the associated functors inside $V_\sigma$. The idea is illustrated by the following diagram.
\begin{figure}[h]
\[
\vcenter{\xymatrix@C=1ex{
			& V_\sigma \\
{\im} A_1 \ar@{^{`}->}[ur]		& {\im} A_2	\ar@{^{`}->}[u] & {\im} A_3 \ar@{_{`}->}[ul]	\\
{\im} A_1 \cap {\im} A_2 \ar@{^{`}->}[u] \ar@{^{`}->}[ur] & {\im} A_1 \cap {\im} A_3 \ar@{_{`}->}[ul] \ar@{^{`}->}[ur] & {\im} A_2 \cap {\im} A_3 \ar@{_{`}->}[ul] \ar@{^{`}->}[u]\\
			& {\im} A_1 \cap {\im} A_2 \cap {\im} A_3 \ar@{_{`}->}[ul] \ar@{^{`}->}[u] \ar@{^{`}->}[ur]}}
\longleftrightarrow
\vcenter{\xymatrix{
V_\sigma \\
{\im} A_1 \ar@{^{`}->}[u]}}
\times
\vcenter{\xymatrix{
V_\sigma \\
{\im} A_2 \ar@{^{`}->}[u]}}
\times
\vcenter{\xymatrix{
V_\sigma \\
{\im} A_3 \ar@{^{`}->}[u]}}
\]
\caption{}
\end{figure}
\end{example}

\begin{example}\label{exampletwo}
Now write the three subspace quiver as $P$ below, and let $Q$ be the extension of $P$ from its sink.
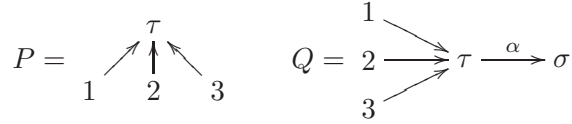
\[
P = \threesubspace{\tau}{1}{2}{3} \qquad Q = \extDfourasleekmaps{1}{2}{3}{\tau}{\sigma}{}{}{}{\alpha}
\]
Given any rank functor $\rank_J$ on $P$ from the previous example, the image functor
\[
V_\alpha (\rank_J(V |_P)) \subseteq V_\sigma
\]
is a rank functor on $Q$.  Now we make the simple observation that for any linear map between vector spaces $A: U \to W$, and two subspace $X, Y \subset U$, the containment $A(X) \cap A(Y) \supseteq A(X \cap Y)$ is not necessarily an equality.  Thus, any \emph{collection} of subsets $\{J_i \} \subseteq B_3$ induces a rank functor on $Q$ given by
\[
V \mapsto \bigcap_i V_\alpha (\rank_{J_i} (V |_P)) \subseteq V_\sigma
\]
which is not in general the image of any one rank functor on $P$.  But because there are inclusions among these rank functors, some collections will be redundant:  for example, if any $J_i = \{1, 2, 3 \}$ then the intersection simplifies to $V_\alpha (\rank_{\{1, 2, 3\}} V)$, since this space is contained in $V_\alpha (\rank_{J'} V)$ for any $J' \subseteq \{1, 2, 3 \}$.  To avoid redundancy, we must consider collections of \emph{incomparable} elements of $B_3$.
\end{example}

These two examples capture the essence of the combinatorics we will use to index a nice set of rank functors on a given rooted tree quiver.

%%%%%%%%%%%%%%%%%%%%%%%%%%%%%%%%%%%%%%%%%%%%%%%
%					COMBINATORIAL CONSTRUCTION						%
%%%%%%%%%%%%%%%%%%%%%%%%%%%%%%%%%%%%%%%%%%%%%%%
\subsection{A combinatorial construction}\label{combinsect}
Before undertaking an analysis of the rank functors on $Q$, we introduce an auxiliary combinatorial framework which will organize the connection between quivers over $Q$, rank functors on $Q$, and the representation ring of $Q$.

First, we recall some definitions which can be found in Stanley's book \cite{stanleyenumcombin}.  An \textbf{(order) ideal} in a poset $A$ is a subset $I \subseteq A$ such that $y \leq x$ and $x \in I$ implies that $y \in I$ also.  In particular, both $\emptyset$ and $A$ are ideals of $A$.
For any subset $\{x_1, \dotsc , x_n \} \subseteq A$, we denote by $\gen{x_1, \dotsc , x_n}$ the smallest ideal of $A$ containing all of the $x_i$.  The set of all order ideals in $A$ is denoted by $J(A)$, and is partially ordered by inclusion.  It is even a distributive lattice, with join operator $\vee$ corresponding to union of ideals and meet operator $\wedge$ corresponding to intersection of ideals.  A map of posets $f\colon A \to B$ induces a map $f \colon J(A) \to J(B)$ which sends an ideal $I \subseteq A$ to the ideal generated by the image $f(I) \subseteq B$.  The \textbf{product} of two posets $A$ and $B$ is their usual product $A \times B$ as sets, ordered by $(x, y) \leq (z, w)$ if and only if both $x \leq z$ and $y \leq w$.  If both $A$ and $B$ are distributive lattices, then so is $A \times B$ with the meet and join operations carried out in each coordinate.
We will always consider $A$ to be a sublattice of $A \times B$ via the inclusion
\begin{equation}\label{latticeproducteqn}
A \simeq \setst{(a, \hat{0})}{a \in A} \subseteq A \times B , 
\end{equation}
and similarly for $B$.  We denote the minimal and maximal elements of a finite lattice $L$ by $\hat{0}$ and $\hat{1}$, respectively, using subscripts to clarify the role of $L$ if necessary.

A set of pairwise incomparable elements $\{x_1, \dotsc, x_k \}$ in a poset $A$ is called an \textbf{antichain} in $A$.  
The map 
\[
\max \colon J(A) \xto{\sim} \{ \text{antichains in } A \}
\]
sending an ideal of $A$ to its set of maximal elements is a bijection between the set of ideals of $A$ and the set of antichains in $A$.  The inverse associates to an antichain $C$ the order ideal generated by $C$.  

Now we proceed to define, for each vertex $x \in Q\verts$, a finite, distributive lattice $\latt{Q}{x}$.  If $Q$ has a single vertex $\sigma$, then $\latt{Q}{\sigma}$ is just the lattice with one element.  For $Q$ with more than one vertex, we define $\latt{Q}{x}$ recursively.  
For any vertex $x \in Q\verts$, there is a unique maximal connected subquiver $Q\tox$ for which $x$ is the sink.  Its vertices are
\[
\left(Q\tox \right)\verts := \setst{y \in Q\verts}{\text{there exists a path from $y$ to $x$}}.
\]
If $x \neq \sigma$, then we can assume that $\latt{Q\tox}{x}$ is already defined, and we take
\[
\latt{Q}{x} := \latt{Q\tox}{x}.
\]
If $x = \sigma$, then removing the vertex $\sigma$ and all arrows attached to $\sigma$ leaves a disjoint union of rooted tree quivers $(Q_i, \sigma_i)$.  We inductively define
\[
\latt{Q}{\sigma} := \prod_i J \left( \latt{Q_i}{\sigma_i} \right) .
\]
In particular, when $(Q, \sigma)$ is an extension of $(P, \tau)$, we have
\[
\latt{Q}{\sigma} = J(\latt{P}{\tau}),
\]
and if $Q$ is obtained by gluing subquivers $P$ and $S$, then we have 
\[
\latt{Q}{\sigma} = \latt{P}{\sigma} \times \latt{S}{\sigma}.
\]
We define the set $L_Q$ as the disjoint union
\[
L_Q := \coprod_{x \in Q\verts} \latt{Q}{x} .
\]

%%%%%%%%%	CONSTRUCTING Q_M ---> Q	%%%%%%%%%%%
\subsection{Reduced quivers over $Q$}\label{reducedquiversect}
Now for each $M \in \latt{Q}{x}$, we will construct a rooted tree quiver $(Q_M, \sigma_M)$ and a map of directed graphs
\[
c_M\colon Q_M \to Q
\]
such that $c_M(\sigma_M) = x$ (and hence $c_M^{-1}(x) = \{ \sigma_M \}$ since $c_M$ preserves heads and tails of arrows).
When $Q$ has one vertex $\sigma$, the lattice $\latt{Q}{\sigma}$ has one element $\hat{1}$ for which we define $Q_{\hat{1}} = Q$ and
\[
c_{\hat{1}} \colon Q \xto{id} Q.
\]
If $Q$ has more than one vertex, we make the definition recursively.  For $M \in \latt{Q}{x}$ such that $x \neq \sigma$, we defined $\latt{Q}{x} = \latt{Q\tox}{x}$, so we can assume that we already have $(Q_M, c_M)$, a quiver over $Q\tox$, which we regard as a quiver over $Q$ via the inclusion
\[
c_M \colon Q_M \to Q\tox \subset Q .
\]
We use the same notation whether considering $Q_M$ as a quiver over $Q$ or $Q\tox$, with the context making the target clear.  For $M \in \latt{Q}{\sigma}$, there are two cases.  As always, we retain the notation from $\S\ref{rootedtreesect}$.

\extcase Let $M \in \latt{Q}{\sigma} = J(\latt{P}{\tau})$ be an order ideal.  If $M = \emptyset$ is the empty ideal, then $M = \hat{0}$ in the lattice $\latt{Q}{\sigma}$, and we take $Q_M$ to have one vertex $\sigma_M$. Define $c_M \colon Q_M \to Q$ to map this vertex to $\sigma$.
Now if $M \neq \emptyset$, let $\max(M) = \{S_1, \cdots S_m\} \subset \latt{P}{\tau}$.  First suppose $m =1$, and set $S:=S_1$.  We define $(Q_{\gen{S}}, \sigma_{\gen{S}})$ as the one point extension of $(P_{S}, \tau_S)$, which we can assume is already defined.  There is a unique map $c_{\gen{S}} \colon Q_{\gen{S}} \to Q$ which extends $c_S \colon P_S \to P$, necessarily sending $\sigma_{\gen{S}}$ to $\sigma$.

For the general case $m \geq 1$, we have already rooted tree quivers $(P_{S_i}, \tau_i)$ and structure morphisms
\[
c_{S_i} \colon P_{S_i} \to P
\]
for each $i$.  Each of these gives 
\[
c_{\gen{S_i}} \colon Q_{\gen{S_i}} \to Q
\]
from the case $m=1$.  We form $Q_M$ by gluing the collection $\{ Q_{\gen{S_i}} \}$ at their sinks $\sigma_{\gen{S_i}}$.  This induces a unique map
\[
c_M \colon Q_M \to Q
\]
which restricts to $c_{\gen{S_i}}$ on each $Q_{\gen{S_i}} \subset Q_M$.

\gluecase
Let $M \in \latt{Q}{\sigma}$.  In the gluing case, by definition we have $\latt{Q}{\sigma} = \latt{P}{\sigma} \times \latt{S}{\sigma}$, so we can write $M =(X, Y)$ for some $X \in \latt{P}{\sigma}$ and $Y \in \latt{S}{\sigma}$.  To define $(Q_M, c_M)$ recursively, we can assume that we have
\[
c_X \colon P_X \to P \qquad \text{and} \qquad c_Y \colon S_Y \to S
\]
defined already.  Then $Q_M$ is given by gluing $P_X$ and $S_Y$ at their sinks, and $c_M \colon Q_M \to Q$ is the unique map which restricts to $c_X$ and $c_Y$ on $P_X$ and $S_Y$, respectively.

\begin{definition}
The quivers $Q_M$ and structure maps $c_M \colon Q_M \to Q$ constructed above for $M \in L_Q$ are the \textbf{reduced} quivers over $Q$.
\end{definition}

We will usually just refer to $Q_M$ alone as a quiver over $Q$, with the structure map $c_M$ being understood.  Also, when $M \in \latt{Q}{\sigma}$, we will sometimes employ a slight abuse of notation by denoting the sink of $Q_M$ also as $\sigma$.  This is unlikely to result in any confusion in the representation theory,  since by definition both $(c_M^* V)_{\sigma_M} = V_\sigma$ and $(c_{M *}W)_\sigma = W_{\sigma_M}$ for any $V \in \repq$ and $W \in \rep(Q_M)$.

Let $\Lambda \xto{\lambda} Q$ and $\Gamma \xto{\gamma} Q$ be any rooted tree quivers over $Q$.  A map of direct graphs $f \colon \Lambda \to \Gamma$ is a \textbf{morphism of quivers over $Q$} if it commutes with the structure maps, that is, if
\[
\xymatrix@C=2ex{
\Lambda \ar[dr]_{\lambda} \ar[rr]^{f} & \ar@{}[d]|{\circlearrowleft} &\Gamma \ar[dl]^{\gamma}\\
& Q }
\]
is a commutative diagram of maps of directed graphs.  We write $f \in \Hom_{\overq} (\Lambda, \Gamma)$.  The following lemma motivates our interest in morphisms between quivers over $Q$ by relating them to rank functors.

\begin{lemma}\label{rankinclusionlemma}
Let $\Lambda \xto{\lambda} Q$ and $\Gamma \xto{\gamma} Q$ be rooted tree quivers over $(Q, \sigma)$, and $f \in \Hom_{\overq} (\Lambda, \Gamma)$.  Assume that $f, \lambda, \gamma$ are all root preserving, and call all of the root vertices of the quivers here $\sigma$ for simplicity.  Then $f$ induces an injective natural transformation $f^* \colon \rank_\Gamma \of \gamma^* \into \rank_\Lambda \of \lambda^*$
such that for any $V \in \repq$ the diagram
\[
\xymatrix{
& (\gamma^*V)_{\sigma} = (\lambda^* V)_{\sigma} = V_\sigma	\\
\rank_{\Gamma} ( \gamma^* V) \ar@{^{`}->}[r]^{f^*} \ar@{^{`}->}[ur]	& \rank_{\Lambda} ( \lambda^* V) \ar@{^{`}->}[u] }
\]
commutes.
\end{lemma}
\begin{proof}
By Theorem \ref{kinserrankthm}, there is a decomposition $\gamma^* V \simeq U \oplus W$ for some subrepresentations $U, W \subseteq \gamma^* V$ such that $U_\sigma = \rank_\Gamma (\gamma^* V)$ and  $U \simeq \id_\Gamma^{\oplus d}$.
This gives a decomposition
\[
\lambda^*V \simeq f^* \gamma^* V \simeq f^*U \oplus f^*W
\] 
which, by additivity of rank functors, gives $\rank_\Lambda (\lambda^*V) = \rank_\Lambda (f^* U) \oplus \rank_\Lambda (f^* W)$.  Since $f^* U \simeq \id_\Lambda^{\oplus d}$, this gives the inclusion
\[
\rank_\Lambda (f^* U) = (f^*U)_\sigma = f^* (U_\sigma) = f^* (\rank_\Gamma (\gamma^* V)) \subseteq \rank_\Lambda (\lambda^*V) . \qedhere
\]
\end{proof}

For elements $M, N$ of a poset $A$, we write $M \prec N$ when $N$ \textbf{covers} $M$ in $A$, which is by definition when $M < N$ and there does not exist any $Z \in A$ with $M < Z < N$.

\begin{prop}\label{covermapprop}
There is an injective map of sets from $L_{Q}$ to the set of quivers over $Q$.  Furthermore, if we fix a vertex $x \in Q\verts$, the following statements hold for any $M, N \in \latt{Q}{x}$.
\begin{enumerate}[(a)]
\item Reduced quivers over $Q$ admit no nontrivial endomorphisms, i.e. $\Hom_{\overq} (Q_M, Q_M) = \{ id \}$.
\item For each cover relation $M \prec N$, there exists a unique morphism $\Hom_{\overq} (Q_M, Q_N)=\{ \rho_{M,N} \}$.
\item If $M \nleq N$, then $\Hom_{\overq} (Q_M, Q_N) = \emptyset$.
\end{enumerate}
In particular, we have that
\[
\Hom_{\overq} (Q_M, Q_N) \neq \emptyset \iff M \leq N .
\]
\end{prop}
\begin{proof}
The map is given by sending $M \in L_{Q}$ to $Q_M$, which can inductively be seen to be injective by gluing and extension.  The three statements about morphisms will be verified by induction on the number of vertices of $Q$.
When $x \neq \sigma$, by definition $Q_M$ and $Q_N$ are also reduced quivers over $Q\tox \subsetneq Q$, so the proposition holds by induction.  So assume $x = \sigma$.

\extcase Let $\max(M) = \{S_1, \dotsc, S_m \}$.  Then $f \in \Hom_{\overq} (Q_M, Q_M)$ restricts to a morphism
\[
\overline{f} \in \Hom_{\downarrow P} \Bigl( \coprod_{i=1}^m P_{S_i}, \coprod_{i=1}^m P_{S_i}\Bigr) = \prod_{i=1}^m \Hom_{\downarrow P} \Bigl( P_{S_i},  \coprod_{i=1}^m P_{S_i}\Bigr)
\]
over $P$.  Since the $S_i$ are pairwise incomparable in $\latt{P}{\tau}$, the induction hypothesis implies that 
\[
\Hom_{\downarrow P}(P_{S_i}, P_{S_j}) =
\begin{cases}
id & i=j \\
\emptyset & i \neq j \\
\end{cases}.
\]
Hence $\overline{f}$ is the identity on $\coprod_{i=1}^m P_{S_i}$, and $f = id_{Q_M}$ is the only way this can extend as a morphism of quivers over $Q$.

For b), suppose that $M \prec N$ in $\latt{Q}{\sigma} = J(\latt{P}{\tau})$, and let $f \in \Hom_{\overq} (Q_M, Q_N)$.  Then there exists $T \in \latt{P}{\tau}$ such that $N = M \cup \{ T \}$ as ideals in $\latt{P}{\tau}$.
First consider the case that $M$ is a principal ideal, say $M = \gen{S}$ for some $S \in \latt{P}{\tau}$, so that $Q_M$ is the one point extension of $P_S$ from its sink.  Since $N$ is an ideal, it must be that either $T \succ S$ or that $T$ and $S$ are incomparable.  If $T \succ S$, then by the induction hypothesis we have a unique morphism $P_S \xto{\rho_{S, T}} P_T$ of quivers over $P$, which must be the restriction of $f$ over $P$.  Since $N = M \cup \{ T \} = \gen{T}$, by construction $Q_N$ is extended from $P_T$, so $f$ is uniquely determined as being the extension of $\rho_{S, T}$ to a morphism $Q_M \to Q_N$ over $Q$.
On the other hand, if $T$ and $S$ are incomparable, then $\max(N) = \{S, T\}$ and so $Q_N$ is glued from $Q_{\gen{S}}$ and $Q_{\gen{T}}$.  Now by considering the restriction of $f$ over $P$ again, the induction hypotheses implies that $f$ must be the inclusion $Q_M = Q_{\gen{S}} \subseteq Q_N$.

Now if $M$ is not principal, say $\max(M) = \{S_1, \dotsc, S_m \}$, then by definition $Q_M$ is obtained by gluing the $Q_{\gen{S_i}}$ at their sinks, so to give a map from $Q_M$ it is enough to define it on each $Q_{\gen{S_i}}$ for $1 \leq i \leq m$.  But for each $i$, again we have either $S_i \prec T$ or $S_i$ and $T$ are incomparable.  If $S_i \prec T$, then necessarily $T \in \max(N)$, so $f$ must restrict to the unique map $Q_{\gen{S_i}} \to Q_{\gen{T}} \subseteq Q_N$ from the principal case.  If they are incomparable, then again $S_i \in \max(N)$ and $Q_{\gen{S_i}}$ is a subquiver of $Q_N$.  In this case $f$ must restrict to the inclusion from $Q_{\gen{S_i}}$ to $Q_N$.

Now suppose $M \nleq N$, and let
\[
\max(M) = \{S_1, \dotsc, S_m\} \qquad \text{and} \qquad \max(N) = \{T_1, \dotsc, T_n \}.
\]
To be compatible with the structure maps, any morphism $f\colon Q_M \to Q_N$ over $Q$ necessarily satisfies $f^{-1}(\sigma_N) = \{\sigma_M\}$, hence would restrict to a morphism
\[
\overline{f} \in \Hom_{\downarrow P} (\coprod_{i=1}^m P_{S_i}, \coprod_{j=1}^n P_{T_j})
\]
over $P$.  But $M \nleq N$ implies that there exists some $i$ such that $S_i \nleq T_j$ for all $j$, so by induction $\Hom_{\downarrow P} (P_{S_i}, P_{T_j}) = \emptyset$ for all $j$.  Hence such an $f$ does not exist over $Q$, and we get $\Hom_{\overq} (Q_M, Q_N) = \emptyset$.

\gluecase We can write $M = (X, Y)$ and $N = (Z, W)$ for some $X, Z \in \latt{P}{\sigma}$ and $Y, W \in \latt{S}{\sigma}$.  By induction, we immediately get (a).
For b), suppose $M \prec N$ and that $f \in \Hom_{\overq} (Q_M, Q_N)$.  Then by switching $P$ and $S$ if necessary, we can assume that $X \prec Z$ and $Y = W$.  Then by induction $f$ must restrict to the unique map $P_X \xto{\rho_{X, Z}} P_Z$ over $P$, and over $S$ to the identity on $S_Y$.  This defines $\rho_{M, N}$ uniquely over $Q$.
Now suppose $M \nleq N$, so without loss of generality $X \nleq Z$.  Any $f \in \Hom_{\overq}(Q_M, Q_N)$ would restrict to some $\overline{f} \in \Hom_{\downarrow P} (P_X, P_Z)$, but by induction such a morphism does not exist.  Hence there is no such $f$, and so $\Hom_{\overq} (Q_M, Q_N) = \emptyset$.
\end{proof}

%%%%%%%%%%%%%%%%%%%%%%%%%%%%%%%%%%%%%%%%%%%%%%%%%%
%						INDUCED RANK FUNCTORS							%
%%%%%%%%%%%%%%%%%%%%%%%%%%%%%%%%%%%%%%%%%%%%%%%%%%
\subsection{Induced rank functors}

For each $M \in \latt{Q}{x} \subset L_Q$, we get a rank functor $\rank_M$ on $Q$ by pulling back a representation along $c_M$, then applying the global rank functor of $Q_M$:
\[
\rank_M V := \rank_{Q_M} \left( c_M^* V \right) \subseteq V_x .
\]
The vector space $\rank_M V$ is naturally a subspace of $V_x$ because Lemma \ref{rootedtreeranklemma} gives an inclusion
\[
\rank_{Q_M} \left( c_M^* V \right) \subseteq (c_M^* V)_{\sigma_M} = V_{c_M (\sigma_M)} = V_x .
\]
We call these subspaces \textbf{rank spaces} of $V$, always considering them as subspaces of some appropriate $V_x$ without explicit mention.

The inductive definition of $\latt{Q}{\sigma}$, along with Lemma \ref{glueextranklemma}, provides the following inductive description of $\rank_M$:  suppose that removing $\sigma$ from $Q$ leaves a disjoint union of rooted trees $(Q_i, \sigma_i)$, so that $M = (M_i)$ for some ideals $M_i \subseteq \latt{Q_i}{\sigma_i}$.  Let $\alpha_i$ be the unique arrow from $\sigma_i$ to $\sigma$.  Then we have that
\[
\rank_M V = \bigcap_i \bigcap_{N \in \max(M_i)} V_{\alpha_i} (\rank_N (V|_{Q_i})) .
\]
%Working with the two steps of gluing and extension serves to simplify our proofs.

\begin{example}
For each vertex $x \in Q\verts$, the reduced quiver $Q_{\hat{0}_x}$ corresponding to the minimal element $\hat{0}_x \in \latt{Q}{x}$ is the inclusion of the vertex $x$ as a subquiver of $Q$, and hence $\rank_{\hat{0}_x}(V) =V_x$.  For the maximal element $\hat{1}_x$, we get $Q_{\hat{1}_x} = Q\tox$ and the structure map is inclusion, hence $\rank_{\hat{1}_x}(V) = \rank_{Q\tox} (V |_{Q\tox})$.
In fact, Lemma \ref{glueextranklemma} and induction imply that for every connected subquiver $P \subseteq Q$, the rank functor $\rank_P$ (applied to the restriction $V|_P$) appears among the functors $\{ \rank_M \}_{M \in L_Q}$.
\end{example}

The relations between induced rank functors given by Lemma \ref{rankinclusionlemma} motivates the following definition.

\begin{definition}
Let $f\colon (Q', \sigma') \to Q$ be a quiver over $Q$.  We will say that $Q' \xto{f} Q$ is \textbf{rank equivalent} to $Q_M$ (for some $M \in L_{Q}$) when there exist morphisms 
\[
Q_M \xto{g} Q' \xto{h} Q_M
\]
as quivers over $Q$ such that $h \circ g = id_{Q_M}$.
\end{definition}

By Lemma \ref{rankinclusionlemma}, if $Q' \xto{f} Q$ is rank equivalent to $Q_M$ then it induces the same rank functor on $Q$:
\[
\rank_{Q'} (f^*(V)) = \rank_M V \subseteq V_{f(\sigma')} .
\]
Furthermore, in such a rank equivalence $g$ is injective on vertices and arrows, so the number of arrows in $Q_M$ is less than or equal to the number of arrows in $Q'$, and if these numbers are equal then $g$ is an isomorphism of quivers over $Q$.  This explains the terminology ``reduced'' quivers over $Q$.
Now to see that for any representation $V$, rank functors give an order reversing map from the lattice $\latt{Q}{x}$ to the collection of subspaces of $V_x$, partially ordered by inclusion.

\begin{prop}\label{rankspaceorderprop}
Let $M, N \in \latt{Q}{x}$ and $V \in \repq$.  Then we have:
\begin{enumerate}[(a)]
\item If $N \geq M$, then $\rank_N V \subseteq \rank_M V$ as subspaces of $V_x$.
\item Join in $\latt{Q}{x}$ corresponds to intersection in $V_x$, i.e.
\[
\rank_{M \vee N} V = \rank_M V \cap \rank_N V . 
\]
\end{enumerate}
\end{prop}
\begin{proof}
Part (a) follows from Proposition \ref{covermapprop} and Lemma \ref{rankinclusionlemma}.
Part (b) is proven by induction.

\extcase Considering $M, N$ as ideals in $\latt{P}{\tau}$, let $\max(M) = \{S_1, \dotsc, S_m \}$ and $\max(N) = \{T_1, \dotsc, T_n \}$.  Then we can apply Lemma \ref{glueextranklemma}, keeping in mind the gluing construction of $Q_M$ and $Q_N$, to get
\[
\rank_M V \cap  \rank_N V = \left( \bigcap_{i=1}^m \rank_{\gen{S_i}} V \right) \cap \left(\bigcap_{i=1}^n \rank_{\gen{T_i}} V \right) . 
\]
By part (a), we only need to intersect over the maximal elements of $\latt{P}{\tau}$ appearing here.  Now since $\gen{S_1, \dotsc , S_m } \cup \gen{T_1 , \dotsc , T_n } = M \vee N$, we get that
\[
\rank_M V \cap  \rank_N V = \bigcap_{A \in \max(M \vee N)} \rank_{\gen{A}} V = \rank_{M \vee N} V .
\]

\gluecase  If we write $M = (X, Y)$ and $N = (Z, W)$ for some $X,Z \in \latt{P}{\sigma}$ and $Y,W \in \latt{S}{\sigma}$, then we have $M \vee N = (X \vee Z,\, Y \vee W)$, so the result follows from Lemma \ref{glueextranklemma} and the induction hypothesis.
\end{proof}

Now we will see that the reduced quivers over $Q$ constructed above give all possible rank functors induced by pullback to a rooted tree.

\begin{theorem}\label{rankequivthm}
Let $(Q, \sigma)$ be a rooted tree, and $f\colon (Q', \sigma') \to Q$ a rooted tree over $Q$. Then there exists $M \in L_Q$ such that  $Q' \xto{f} Q$ is rank equivalent to $Q_M$.
\end{theorem}

\begin{proof}
The idea is to use induction with Lemma \ref{glueextranklemma}.  When $Q$ has one vertex, a representation of $Q$ is a vector space and the identity functor is the only rank functor on $Q$.
If $Q$ has more than one vertex, assume that the proposition holds for any quiver with fewer vertices.
By induction on the number of vertices of $Q$, we can reduce to the case that $f(\sigma') = \sigma$.

\extcase First, assume that $Q'$ is a one point extension of $(P', \tau')$ by an arrow $\alpha'$, so $f(\sigma') = \sigma$ implies that $f(\alpha') = \alpha$ and $f(\tau') = \tau$.
The induction hypothesis gives $S \in \latt{P}{\tau}$ and morphisms of quivers over $P$
\[
P_S \xto{\tilde{g}} P' \xto{\tilde{h}} P_S
\]
such that $\tilde{h} \circ \tilde{g} = id$.
These morphisms uniquely extend to maps $Q_{\gen{S}} \xto{g} Q' \xto{h} Q_{\gen{S}}$ over $Q$, satisfying $h \circ g = id$.

Now an arbitrary $Q'$ can be written as $Q' = \coprod_{i \in I}^\sigma Q'_i$, where each $Q'_i$ is a one point extension of some $(P'_i, \tau'_i)$.  Using the previous case, for each $i \in I$ we have $S_i \in \latt{P}{\tau}$ and morphisms of quivers over $Q$
\[
Q_{\gen{S_i}} \xto{\tilde{g}_i} Q'_i \xto{\tilde{h}_i} Q_{\gen{S_i}}
\]
such that $\tilde{h_i} \circ \tilde{g_i} = id$.  The elements $\{S_i\}_{i \in I}$ generate an ideal $M \subseteq \latt{P}{\tau}$ which we claim corresponds to the desired reduced quiver over $Q$.
The ideal $M$ has maximal elements $\max(M) = \{S_i \}_{i \in J}$ for some (not necessarily unique) subset $J \subseteq I$.  For $i \in J$, define $\rho_i \colon Q_{\gen{S_i}} \xto{id} Q_{\gen{S_i}}$.
For $i \notin J$ we can choose some (not necessarily unique) $j(i) \in J$ such that $S_i \leq S_{j(i)}$, and so by Proposition \ref{covermapprop} there exists a morphism $\rho_i \colon Q_{\gen{S_i}} \to Q_{\gen{S_{j(i)}}}$ of quivers over $Q$.   The collection $\{\rho_i\}_{i \in I}$ induces a map
\[
\rho\colon {\coprod_{i \in I}}^\sigma Q_{\gen{S_i}} \to  {\coprod_{i \in J}}^\sigma Q_{\gen{S_i}} = Q_M
\]
of quivers over $Q$.  Then we get a commutative diagram of quivers over $Q$
\[
\xymatrix{
{\coprod}^\sigma_{i \in I} Q_{\gen{S_i}}	 \ar[r]^-{\coprod^\sigma \tilde{g}_i}	& Q' \ar[r]^-{\coprod^\sigma \tilde{h}_i} \ar[dr]		& {\coprod}^\sigma_{i \in I} Q_{\gen{S_i}} \ar[d]^{\rho}	\\
Q_M \ar@{^{`}->}[u]	 \ar[rr]^{id} \ar[ur] & 						& Q_M		\\
}
\]
with the lower triangle giving a rank equivalence between $Q'$ and $Q_M$.

\gluecase
If $Q$ is a gluing of $P$ and $S$, then $Q'$ is a gluing of $P' := f^{-1}P$ and $S' := f^{-1}S$. By restricting $f$, we see that $P'$ and $S'$ are quivers over $P$ and $S$, respectively.  Then by the induction hypothesis, we have $X \in \latt{P}{\sigma}$ and $Y \in \latt{S}{\sigma}$ and morphisms
\[
P_X \xto{\tilde{g}_P} P' \xto{\tilde{h}_P} P_X
\qquad \text{and} \qquad
S_Y \xto{\tilde{g}_S} S' \xto{\tilde{h}_S} P_Y
\]
that satisfy $\tilde{h}_P \circ \tilde{g}_P = id$ and $\tilde{h}_S \circ \tilde{g}_S = id$.
These induce a rank equivalence between $Q'$ and $Q_M=Q_{(X, Y)}$.
\end{proof}

This shows that there are only finitely many distinct rank functors on $Q$ induced by global rank functors of rooted trees over $Q$.

%%%%%%%%%%%%%%%%%%%%%%%%%%%%%%%%%%%%%%%%%%%%%%%
%					 EXAMPLES EXAMPLES EXAMPLES						%
%%%%%%%%%%%%%%%%%%%%%%%%%%%%%%%%%%%%%%%%%%%%%%%
\begin{example}\label{nsubspaceeg}
Generalizing Example \ref{exampleone}: let $(Q, \sigma)$ be the $n$-subspace quiver, labeled as
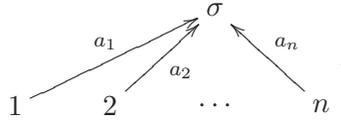
\[
\nsubspacemaps{1}{2}{n}{\sigma}{a_1}{a_2}{a_n} . 
\]
Then $\latt{Q}{x}$ has one element for $x \neq \sigma$, and $\latt{Q}{\sigma}$ is the lattice $B_n$ of subsets of $\{1, \dotsc, n \}$.  The reduced quivers over $Q$ are exactly the connected subquivers of $Q$.  The rank functor corresponding to $J \subseteq \{1, \dotsc, n \}$ sends $V \in \repq$ to
\[
\rank_J V = \bigcap_{j \in J} \im V_{a_j} . 
\]
%The reduced representations of $Q$ are exactly the indecomposable idempotent representations of $Q$, which are those indecomposables with dimension 0 or 1 at each vertex.
\end{example}

\begin{example}\label{latticeseg}
Continuing with $Q$ as in Example \ref{exampletwo}, we have that $\latt{Q}{\tau} = \latt{P}{\tau} = B_3$, and $\latt{Q}{\sigma} = J(B_3)$.  The Hasse diagrams (with smaller elements drawn towards the top) are illustrated in Figure ~\ref{fig:bigfig} on page ~\pageref{fig:bigfig}.
\begin{figure}
\[
\latt{Q}{\tau} = \vcenter{\xymatrix{
			& {\tau}	\ar@{-}[dl] \ar@{-}[d] \ar@{-}[dr] \\
{1\tau} \ar@{-}[d] \ar@{-}[dr]		& {2\tau}	\ar@{-}[dl] \ar@{-}[dr] & {3\tau} \ar@{-}[d] \ar@{-}[dl]\\
{12\tau} \ar@{-}[dr]& {13\tau} \ar@{-}[d] & {23\tau} \ar@{-}[dl] \\
			& {123\tau} }}
\qquad
\latt{Q}{\sigma} = \vcenter{\xymatrix{
			& {\sigma}	\ar@{-}[d] \\
			& {\tau \sigma}	\ar@{-}[dl] \ar@{-}[d] \ar@{-}[dr] \\
{1\tau \sigma} \ar@{-}[d] \ar@{-}[dr]	& {2\tau \sigma} \ar@{-}[dl] \ar@{-}[dr] & {3\tau \sigma} \ar@{-}[d] \ar@{-}[dl]\\
{\bullet} \ar@{-}[d] \ar@{-}[drrr]	& {\bullet}	\ar@{-}[d]	\ar@{-}[drr] & {\bullet} \ar@{-}[d] \ar@{-}[dr]\\
{12\tau \sigma} \ar@{-}[d]		& {13\tau \sigma}	\ar@{-}[d]		& {23\tau \sigma}	\ar@{-}[d]		& {\bullet}	\ar@{-}[dlll] \ar@{-}[dll] \ar@{-}[dl] \\
{\bullet} \ar@{-}[d] \ar@{-}[dr] & {\bullet}	\ar@{-}[dl] \ar@{-}[dr] & {\bullet} \ar@{-}[d] \ar@{-}[dl]\\
{\bullet} \ar@{-}[dr]& {\bullet} \ar@{-}[d] & {\bullet} \ar@{-}[dl] \\
			& {\bullet}	\ar@{-}[d] \\
			& {123\tau \sigma}	}} . 
\]
\caption{}\label{fig:bigfig}
\end{figure}
The elements of the lattices that have been labeled are those whose corresponding rank functor is the global rank functor of some subquiver of $Q$.  The label is then the set of vertices in the corresponding subquiver.
\end{example}

It is interesting to note that the lattices $\latt{Q}{x}$ are always self-dual, which follows easily from the inductive definition.

%%%%%%%%%%%%%%%%%%%%%%%%%%%%%%%%%%%%%%%%%%%%%%%%%%
%						REDUCED REPRESENTATIONS					%
%%%%%%%%%%%%%%%%%%%%%%%%%%%%%%%%%%%%%%%%%%%%%%%%%%
\section{Reduced Representations of $Q$}
\label{reducedrepssect}

\subsection{Construction and first properties}
We turn our focus to studying a set of representations of $Q$, also indexed by $L_Q$, which are in some sense dual to the rank functors.% $r_M \colon R(Q) \to \Z$.

\begin{definition}
Let $Q$ be a rooted tree quiver and $M \in \latt{Q}{x}$.  Define a representation of $Q$ by
\[
\charrep_M := c_{M*} \id_{Q_M} 
\]
which will be called a \textbf{reduced representation} of $Q$.
\end{definition}

Such a representation is by definition a \textbf{tree module} in the sense of \cite{ringelexceptionalmodulestree}, defined over $\Z$ by 0-1 matrices.
In general, let $Q' \xto{c} Q$ be a quiver over $Q$, and $\{v_y\}_{y \in Q'\verts}$ a vector space basis for $\id_{Q'}$ such that $(\id_{Q'})_a (v_{ta}) = v_{ha}$ for every arrow $a$ of $Q'$.  Then this gives the representation $c_* \id_{Q'}$ a vector space basis $\{v_y \}_{y \in Q'\verts}$ such that at each vertex $x \in Q\verts$, we have
\[
\left( c_* \id_{Q'} \right)_x = \bigoplus_{y \in c_M^{-1}(x)} Kv_y . 
\]
We call such a basis $\{v_y\}$ a \textbf{standard basis} for $c_* \id_{Q'}$.

It is natural to ask what kind of morphisms exist between these reduced representations.  We saw in Proposition \ref{covermapprop} that the existence of morphisms between reduced quivers over $Q$, whose sinks lie over a common vertex $x \in Q\verts$, correspond to relations in the poset $\latt{Q}{x}$. One might wonder if the same phenomenon holds for reduced representations of $Q$, that is, if all homomorphisms between reduced representations are induced from morphisms of quivers over $Q$. This is not exactly true.  The basic problem with this comes from homomorphisms which are 0 at $x$, the sink of their supports. To get a nice correspondence we must ``stabilize'' the hom spaces, as defined below.

\begin{definition}
For $M, N \in \latt{Q}{x}$, the \textbf{stable hom space} between $\charrep_M$ and $\charrep_N$ is 
\[
\stabHom_Q (\charrep_M, \charrep_N) := \frac{\Hom_Q (\charrep_M, \charrep_N )}{ \setst{f}{f_x =0}} .
\]
\end{definition}

For a set $S$, let $K\gen{S}$ denote the free vector space on the elements of $S$.

\begin{theorem} \label{charrephomthm}
%For $M, N \in \latt{Q}{x}$, a map $f \colon Q_M \to Q_N$ of quivers over $Q$ induces a morphism $\tilde{f} \colon \charrep_M \to \charrep_N$ in $\repq$, such that $\tilde{f}_x \neq 0$.  
Let $\Lambda \xto{\lambda} Q$ and $\Gamma \xto{\gamma} Q$ be rooted tree quivers over $Q$.  A root preserving map $f \colon \Lambda \to \Gamma$ of quivers over $Q$ induces a morphism $\tilde{f} \colon \lambda_* \id_{\Lambda} \to \gamma_* \id_{\Gamma}$ in $\repq$ such that $\tilde{f}_x \neq 0$, where $x = \lambda (\sigma_\Lambda) = \gamma (\sigma_\Gamma)$.
This gives, for $M, N \in \latt{Q}{x}$, a surjective map of vector spaces 
\[
K \gen{ \Hom_{\overq}(Q_M, Q_N)} \onto \stabHom_Q (\charrep_M , \charrep_N) . 
\]
In particular, $\stabHom_Q (\charrep_M, \charrep_N) \neq 0$ if and only if $M \leq N$.
\end{theorem}
\begin{proof}
Let $\{v_y \}$ and $\{w_y \}$ be standard bases for $\lambda_* \id_\Lambda$ and $\gamma_* \id_\Gamma$, respectively.  Define $\tilde{f}(v_y) := w_{f(y)}$, which gives a map of vector spaces from $\lambda_* \id_\Lambda$ to $\gamma_* \id_\Gamma$.  To see that this is a morphism of quiver representations, suppose that $a\in Q\arrows$ is an arrow in $Q$, and fix $y \in  \lambda^{-1}(ta)$.  We need to show that 
\[
\tilde{f} \circ (\lambda_* \id_\Lambda)_a (v_y) = (\gamma_* \id_\Gamma)_a \circ \tilde{f} (v_y) . 
\]
The vertex $ta$ is not a sink, so the assumption that $f$ is root preserving implies that $y$ is not a sink either.  Then since $\Lambda$ is a rooted tree, there is a unique arrow $b$ with $tb = y$.  Then by definition, we have $\tilde{f} \circ (\lambda_* \id_\Lambda)_a (v_y) = \tilde{f} (v_{hb}) = w_{f(hb)}$.
On the other hand, $\tilde{f} (v_y) = w_{f(y)}$ is by definition.  Similarly to the situation above, there is a unique arrow in $\Gamma$ with tail $f(y)$, and this arrow must be $f(b)$ since $b$ has nowhere else to map to.  The head of $f(b)$ must be $f(hb)$ since $f$ is a map of directed graphs.  Hence we have $(\gamma_* \id_\Gamma)_a (w_{f(y)}) = w_{f(hb)}$, so $\tilde{f}$ is a map of quiver representations.  To simplify the notation, we will drop the tilde throughout the rest of proof.
Since $f$ is root preserving, it takes the sink $\sigma_\Lambda$ of $\Lambda$ to the sink $\sigma_\Gamma$ of $\Gamma$.  Then $f_x (v_{\sigma_\Lambda}) = w_{\sigma_\Gamma}$ holds by definition, so $f_x$ is nonzero.

Now take $\Lambda = Q_M$ and $\Gamma = Q_N$ to be reduced quivers over $Q$, with $\{v_y \}$ and $\{w_y \}$ still standard bases as above.  We show that the induced map is surjective, by induction.  Given $\overline{f}$ in the stable hom space, choose a representative $f \in \Hom_Q (\charrep_M, \charrep_N)$ of $\overline{f}$.  Define $\kappa \in K$ by $f_\sigma (v_\sigma) = \kappa w_\sigma$, which does not depend on the representative $f$ chosen.  We can assume $\kappa \neq 0$.

\extcase Let $\max(M) = \{S_1, \dotsc, S_m \}$ and $\max(N) = \{T_1, \dotsc, T_n \}$.  Then $f$ restricts over $P$ to
\[
f|_P = \sum_{ij} f_{ij} \in \Hom_P \left( \bigoplus_i \charrep_{S_i}, \bigoplus_j \charrep_{T_j} \right)
\]
with each $f_{ij} \in \Hom_P (\charrep_{S_i}, \charrep_{T_j} )$.  By induction, there exists for each pair $(i, j)$ a collection
\[
g_{ij}^k \in \Hom_{\downarrow P} (P_{S_i}, P_{T_j} ) \qquad 1 \leq k \leq d(i,j)
\]
and scalars $\lambda_{ij}^k \in K$ such that
\[
\overline{f_{ij}} = \sum_{k=1}^{d(i,j)} \lambda_{ij}^k \overline{g_{ij}^k} . 
\]
Standard bases of $\charrep_M$ and $\charrep_N$ restrict to standard bases of each $\charrep_{S_i}$ and $\charrep_{T_j}$.  After normalizing the maps induced by the $g_{ij}^k$ to be compatible with these standard bases at $\tau$, we can assume that
\begin{equation}\label{lambdasumeqn}
\sum_j \sum_{k = 1}^{d(i,j)} \lambda_{ij}^k = \kappa . 
\end{equation}
We can identify $\Hom_{\overq} (Q_M, Q_N) = \Hom_{\downarrow P} (\coprod_i P_{S_i}, \coprod_j P_{T_j})$, because any element of the right hand side extends uniquely over the extending arrow $\alpha$ to a morphism on the left hand side.
Thus, we can give an element of $\Hom_{\overq} (Q_M, Q_N)$
in terms of quivers over $P$ by specifying an $m$-tuple $(j_1, \dotsc, j_m)$ with $1 \leq j_i \leq n$, and a sequence of $k_i$'s with $1 \leq k_i \leq d(i,j_i)$
\[
(g_{1 j_1}^{k_1}, \dotsc, g_{m j_m}^{k_m}) \in \Hom_{\overq} (Q_M, Q_N) . 
\]

We will show that
\begin{equation}\label{claimedfeqn}
\overline{f} = \kappa^{1-m} \sum_{(j_1, \dotsc, j_m)} \sum_{(k_1, \dotsc, k_m)} \left[ \left( \prod_{i=1}^m \lambda_{i j_i}^{k_i} \right) \overline{( g_{1 j_1}^{k_1}, \dotsc, g_{m j_m}^{k_m})} \right]
\end{equation}
where the first two sums are indexed as above.  The key is to see that
\begin{equation}\label{gammameqn}
\sum_{(j_1, \dotsc, j_m)} \sum_{(k_1, \dotsc, k_m)} \left( \prod_{i=1}^m \lambda_{i j_i}^{k_i} \right) = \kappa^m
\end{equation}
holds in the field $K$, which can be shown by an easy induction on $m$, using the assumption of equation (\ref{lambdasumeqn}). Now over $P$, the map induced on representations by $( g_{1 j_1}^{k_1}, \dotsc, g_{m j_m}^{k_m})$ restricts to the sum 
\[
\sum_i g_{i j_i}^{k_i} \in \Hom_P \left( \bigoplus_i \charrep_{S_i}, \bigoplus_j \charrep_{T_j} \right) . 
\]
Using this we can compute the coefficient of $\overline{g_{ab}^c}$ in the restriction of the right hand side of equation (\ref{claimedfeqn}) to $P$.  By factoring out $\lambda_{ab}^c$, which always appears when $\overline{g_{ab}^c}$ does, we get
\[
\kappa^{1-m} \lambda_{ab}^c \sum_{(j_1, \dotsc, \hat{j_a}, \dotsc, j_m)} \sum_{(k_1, \dotsc, \hat{k_a}, \dotsc, k_m)} \prod_{i \neq a} \lambda_{i j_i}^{k_i} = \kappa^{1-m} \lambda_{ab}^c \kappa^{m-1} = \lambda_{ab}^c
\]
where $\hat{j_a}$ and $\hat{k_a}$ mean to omit those indices. Then the next to last equality follows from equation (\ref{gammameqn}).  Hence the right hand side of equation (\ref{claimedfeqn}) agrees with $\overline{f}$ over $P$.  But at $\sigma$, each map $( g_{1 j_1}^{k_1}, \dotsc, g_{m j_m}^{k_m})$ sends $v_\sigma$ to $w_\sigma$, and so the right hand side of equation (\ref{claimedfeqn}) maps $v_\sigma$ to
\[
\kappa^{1-m} \sum_{(j_1, \dotsc, j_m)} \sum_{(k_1, \dotsc, k_m)} \left( \prod_{i=1}^m \lambda_{i j_i}^{k_i} \right) w_\sigma = \kappa^{1-m} \kappa^m w_\sigma = \kappa w_\sigma
\]
which agrees with $\overline{f}$ also.  So equation (\ref{claimedfeqn}) holds, expressing $\overline{f}$ as a linear combination of morphisms induced by maps of quivers.

\gluecase  Let $M = (X, Y)$ and $N=(Z, W)$.  By induction, we can write
\[
\overline{f} |_P = \sum_i \lambda_i \overline{g_i} \qquad \qquad \overline{f}|_S = \sum_j \mu_j \overline{h_j} \qquad \lambda_i, \mu_j \in K
\]
for some collection of $g_i \in \Hom_{\downarrow P} (P_X, P_Z)$ and $h_j \in \Hom_{\downarrow S} (S_Y, S_W)$.  Note that, since each $g_i$ and $h_j$ sends $v_\sigma$ to $w_\sigma$, we have that
\[
\sum_i \lambda_i = \sum_j \mu_j = \kappa
\]
in order for the restrictions over $P$ and $S$ to be equal at $\sigma$.  Let
\[
(g_i, h_j) \in \Hom_{\overq} (Q_M, Q_N) = \Hom_{\downarrow P} (P_X, P_Z) \times  \Hom_{\downarrow S} (S_Y, S_W)
\]
be the morphism of quivers over $Q$ defined by $g_i$ over $P$ and $h_j$ over $S$.  Then it is straightforward to check that
\[
\overline{f} = \kappa^{-1} \sum_{ij} \lambda_i \mu_j \overline{(g_i, h_j)}
\]
which completes the proof.
\end{proof}

One may note the similarity in spirit of this theorem to a theorem of Crawley-Boevey \cite{crawleyboeveyzerorelation} on morphisms between tree modules over zero-relation algebras.
As a corollary, we get an alternative characterization of reduced quivers over $Q$.

\begin{corollary}\label{indecompcor}
If $Q' \xto{c} Q$ is a rooted tree quiver over $(Q, \sigma)$, then $c_* \id_{Q'}$ is indecomposable if and only if $Q'$ is a reduced quiver over $Q$.  Thus, the reduced representations $\charrep_M$ are indecomposable and pairwise non-isomorphic.
%and any $f \in \End_Q (\charrep_M)$ such that $f_\sigma = 0$ is nilpotent.
\end{corollary}
\begin{proof}
Suppose that $Q' \xto{c} Q$ is not reduced.  Then by Theorem \ref{rankequivthm}, there is some $M \in L_Q$ and there are maps $Q_M \xto{g} Q' \xto{h} Q_M$ of quivers over $Q$ such that $h \circ g = id$.  By Theorem \ref{charrephomthm}, this induces $\charrep_M \xto{\tilde{g}} c_* \id_{Q'} \xto{\tilde{h}} \charrep_M$ such that $\tilde{h} \circ \tilde{g} = id$, and so $\charrep_M$ is a direct summand of $c_* (\id_{Q'})$.

Now suppose that $Q' = Q_M$ is reduced.  To prove that $\charrep_M$ is indecomposable, we use induction and the theorem.

\extcase  Let $\max(M) = \{S_1, \dotsc, S_m \}$.  If $\charrep_M \simeq V \oplus W$ for some subrepresentations $V, W \subseteq \charrep_M$, then since $\dim_K (\charrep_M)_\sigma = 1$, without loss of generality we can assume $W_\sigma = 0$.  By construction, there is a decomposition $\charrep_M |_P \simeq \bigoplus_{i=1}^m U_i$ for some subrepresentations $U_i \simeq \charrep_{S_i}$, and by the induction hypothesis each of these summands is indecomposable.  Using a standard basis for $\charrep_M$, we can even take these $U_i$ such that $(\charrep_M)_\alpha ((U_i)_\tau) \neq 0$ for all $i$.  
But using the decomposition $\charrep_M \simeq V \oplus W$, and the fact that the indecomposable summands are uniquely determined, we can also find a decomposition $\charrep_M |_P \simeq \bigoplus_i X_i$ for some subrepresentations $X_i \simeq \charrep_{S_i}$, such that (after perhaps renumbering) $X_i \subseteq V|_P$ for $1 \leq i \leq k$, and $X_i \subseteq W|_P$ for $k+1 \leq i \leq m$.

Then an isomorphism $\charrep_M \simeq V \oplus W$ would give a commutative diagram 
\[
\vcenter{\xymatrix{
{\bigoplus}_i (U_i)_\tau \ar[r]^{A} \ar[d]^{B}	& K \ar[d]^{C}	\\
{\bigoplus}_i (X_i)_\tau \ar[r]^{D} 	& K} }
\]
over the extending arrow $\alpha$ such that both vertical maps are isomorphisms.
The theorem implies that, since the elements of $\{S_1, \dotsc, S_m \}$ are pairwise incomparable,
\[
\stabHom_P (\charrep_{S_i}, \charrep_{S_j}) =
\begin{cases}
K & i = j \\
0 & i \neq j \\
\end{cases} .
\]
Hence the matrix giving $B$ is diagonal, say with entries $\lambda_i$.  But since $W$ is a direct summand and $W_\sigma = 0$, we get $D \circ B \left( (U_{k+1})_\tau \right) = D \left( (X_{k+1})_\tau \right) = 0$, whereas $C \circ A \left( (U_{k+1})_\tau \right) = C \left( K \right) = K$.  Since the diagram is supposed to commute, this is a contradiction; hence no nontrivial direct sum decomposition of $\charrep_M$ exists.

\gluecase If $\charrep_M \simeq V \oplus W$ is a nontrivial decomposition, then it must restrict to a nontrivial decomposition over either $P$ or $S$.  But $\charrep_M$ restricts to some reduced representation on both $P$ and $S$, which by induction is indecomposable.  Hence $\charrep_M$ has no nontrivial decomposition.  Now pairwise non-isomorphic follows since $\stabHom_Q (\charrep_M, \charrep_N) \neq 0$ implies that $\stabHom_Q (\charrep_N, \charrep_M) = 0$ for $M \neq N$.
\end{proof}

%%%%%%%%%%%		ADJUNCTIONS	%%%%%%%%%%%%%%%%%%
\subsection{Combinatorial adjunctions and reduced representations}
Our goal is to gain some understanding of the structure of the representation ring $R(Q)$ through the representation rings $R(Q_M)$.  We have seen that the order relations in the lattices $\latt{Q}{\sigma}$ encode a lot of information about morphisms between quivers over $Q$, and morphisms between the reduced representations of $Q$.  So in order to connect the representation theory of $Q$ and $Q_M$, it is natural to seek some combinatorial connection between the lattices $\latt{Q}{\sigma}$ and $\latt{Q_M}{\sigma}$.  Summarily, we will see that for any $M \in \latt{Q}{\sigma}$, there is an \emph{adjunction} (sometimes called a \emph{Galois connection}) between the lattices $\latt{Q_M}{\sigma}$ and $\latt{Q}{\sigma}$.
%, which can be found in \cite{primergaloisconnections, adjunctionsgaloisconnections}. 
Simply put, an \textbf{adjunction} is a pair of maps
\[
\xymatrix{ A \ar@<0.5ex>[r]^{\lambda} & B \ar@<0.5ex>[l]^{\rho}}
\]
between posets $A$ and $B$, such that 
\[
a \leq \rho(b) \iff \lambda (a) \leq b
\]
holds for all $a \in A, b \in B$.  We say that $(\lambda, \rho)$ are an adjoint pair with $\lambda$ the lower adjoint, and $\rho$ the upper adjoint.  The following proposition extracts from \cite[Prop.~3,4]{primergaloisconnections} a summary of what we'll need to use from the theory of adjunctions.

\begin{prop}
\label{adjunctionprop}
Suppose $\lambda \colon A \to B$ is any order preserving map between finite lattices.  If $\lambda$ preserves joins, then it has a unique upper adjoint given by
\begin{equation}\label{rightadjointeqn}
\rho ( b) = \max \setst{a \in A}{\lambda(a) \leq b}
\end{equation}
which, furthermore, preserves meets.
\end{prop}

Now fix $M \in \latt{Q}{\sigma}$.  Then for any $A \in \latt{Q_M}{\sigma}$, the composition of structure maps 
\[
(Q_M)_A \xto{c_A} Q_M \xto{c_M} Q 
\]
gives a quiver over $Q$.  By Theorem \ref{rankequivthm}, this quiver over $Q$ is rank equivalent to a unique reduced quiver over $Q$, which we'll denote by $Q_{\pi_* (A)}$.  This gives a map of sets
\[
\pi_* \colon \latt{Q_M}{\sigma} \to \latt{Q}{\sigma}
\]
such that
\[
\rank_{\pi_*(A)} = \rank_A \of c^*_M
\]
for $A \in \latt{Q_M}{\sigma}$.

\begin{prop}\label{quivadjunctionprop}
Let $(Q, \sigma)$ be a rooted tree quiver and $M \in \latt{Q}{\sigma}$.  Then there is an adjunction
\[
\xymatrix{ \latt{Q_M}{\sigma} \ar@<.5ex>[r]^{\pi_*}  & \ar@<.5ex>[l]^{\pi^*} \latt{Q}{\sigma} } .
\]
\end{prop}
\begin{proof}
To simplify the notation, let $c:= c_M$, $L' := \latt{Q_M}{\sigma}$, and $L := \latt{Q}{\sigma}$. If $A \leq B$ in $L'$, then there exists a map $\rho \colon (Q_M)_A \to (Q_M)_B$ of quivers over $Q_M$, by Proposition \ref{covermapprop}.  Composing with the structure map $c$ gives a map between them as quivers over $Q$, so using the definition of rank equivalence we get a sequence of maps over $Q$
\[
Q_{\pi_* (A)} \to (Q_M)_A \xto{\rho} (Q_M)_B \to Q_{\pi_* (B)}
\]
and hence $\pi_* (B) \geq \pi_* (A)$.  So $\pi_*$ is order preserving. Now we can simply calculate
%\begin{multline*}
\[
\rank_{\pi_*(A \vee B)} = \rank_{A \vee B} \of c^* = \rank_A \of c^* \cap \rank_B \of c^* \\
= \rank_{\pi_* (A)} \cap \rank_{\pi_* (B)} = \rank_{\pi_*(A) \vee \pi_* (B)}
\]
%\end{multline*}
which shows that $\pi_*$ preserves the join operation.  By Proposition \ref{adjunctionprop},
\[
\pi^*(N) = \max \setst{A}{\pi_*(A) \leq N}
\]
is the upper adjoint to $\pi_*$.
\end{proof}

One can show that $(Q_M)_{\pi^* (N)}$ is the fiber product of $Q_N$ and $Q_M$ over $Q$.  This fact won't be proven, however, since it won't be needed in the paper.
Now in order to use this adjunction to inductively study the representation theory of $Q$, we need show that it is compatible with gluing and extension in an appropriate sense.

\begin{lemma}\label{adjunctioncompatlemma}
Fix $M \in \latt{Q}{\sigma}$ and let $(\pi_*, \pi^*)$ be the adjunction of Proposition \ref{quivadjunctionprop}. Suppose $Q$ is obtained from $P$ by extension, and that $\max (M) = \{S_1, \dotsc, S_m \} \subseteq \latt{P}{\tau}$.  For each $1 \leq i \leq m$, let $(\pi^i_*, \pi_i^*)$ be the adjunction of Proposition \ref{quivadjunctionprop} corresponding to $(P, \tau)$ and $S_i$.  Then, if we write $A  \in \latt{Q_M}{\sigma}$ as a product of ideals $(A_1, \dotsc, A_m) \in {\prod}_{i=1}^m J(\latt{P_{S_i}}{\tau})$,  we have
\begin{equation}\label{extpusheqn}
\pi_* (A) = \bigcup_{i=1}^m \pi_*^i (A_i) 
\end{equation}
(on the right hand side $\pi_*^i$ is the induced map on ideals as in \S\ref{combinsect}).
Similarly, for $N \in \latt{Q}{\sigma}$ we get
\begin{equation}\label{extpulleqn}
\pi^* (N) = (\pi^*_1 (N), \dotsc, \pi^*_m (N))
\end{equation}
where $N$ is considered as an ideal of $\latt{P}{\tau}$ on the right hand side.

When $Q$ is glued together from $P$ and $S$, and we write $M = (X, Y) \in \latt{P}{\sigma} \times \latt{S}{\sigma}$, then we have adjunctions $(\pi_*^P, \pi^*_P)$ and $(\pi_*^S, \pi^*_S)$ over $P$ and $S$, respectively.  In this case, for $(A, B) \in \latt{P_X}{\sigma} \times \latt{S_Y}{\sigma} = \latt{Q_M}{\sigma}$ we have
\begin{equation}\label{gluepusheqn}
\pi_* ((A, B)) = (\pi_*^P (A), \pi_*^S (B) ) , 
\end{equation}
and similarly, for $N = (Z, W)$ we find
\begin{equation}\label{gluepulleqn}
\pi^* (N) = (\pi^*_P (Z), \pi^*_S (W) ) .
\end{equation}
\end{lemma}
\begin{proof}
As usual, we proceed by induction, retaining the notation from the statement of the lemma.

\extcase
First consider the case that $M = \gen{S}$ is principal, so $Q_M$ is a one point extension of $(P_S, \tau)$, and $\latt{Q_M}{\sigma} = J(\latt{P_S}{\tau})$.
For any principal ideal $\gen{T} \subseteq \latt{P_S}{\tau}$, we know $\rank_{\gen{T}} V = V_\alpha (\rank_T V|_{P_S})$ from Lemma \ref{glueextranklemma}, so we can simply compute
\[
\rank_{\pi_* (\gen{T})}V = \rank_{\gen{T}}(c_M^* V) = V_\alpha (\rank_T (c_S^* (V|_P))) = V_\alpha (\rank_{\pi_*^1 (T)} (V|_P)) = \rank_{\gen{\pi_*^1 (T)}} V
\]
which shows that $\pi_* (\gen{T}) = \gen{\pi_*^1 (T)}$.  
We know that $\pi_*$ commutes with join and $\pi_*^P$ commutes with union of ideals, and these two operations correspond with one another in the identification $\latt{Q}{\sigma} = J(\latt{P}{\tau})$. Thus equation (\ref{extpusheqn}) holds for an arbitrary element of $\latt{Q_M}{\sigma}$ when $M$ is principal.

In general, suppose that $\max(M) = \{S_1, \dotsc, S_m \}$, so we have the identification
\[
\latt{Q_M}{\sigma} = \prod_{i=1}^m \latt{Q_{\gen{S_i}}}{\sigma} . 
\]
Under natural embeddings $\latt{Q_{\gen{S_i}}}{\sigma} \into \latt{Q_M}{\sigma}$ of (\ref{latticeproducteqn}), every element of $\latt{Q_M}{\sigma}$ can be written as a join of elements in the images of these, so again the correspondence between join and union shows that equation (\ref{extpusheqn}) holds in general.

\gluecase 
For any $(A, B) \in \latt{Q_M}{\sigma} = \latt{P_X}{\sigma} \times \latt{S_Y}{\sigma}$, we can compute
%\begin{multline*}
\[
\rank_{\pi_* (A, B)} = \rank_{(A, B)} \of c^* = \rank_A \of c^*|_P \cap  \rank_B \of c^*|_S
= \rank_{\pi^P_*(A)} \cap \rank_{\pi^S_*(B)} = \rank_{(\pi^P_* (A), \pi^S_* (B))}
\]
%\end{multline*}
which shows that $(\pi^P_*(A), \pi^S_*(B)) = \pi_*(A, B)$.  

A routine argument using uniqueness of upper adjoints from Proposition \ref{adjunctionprop} gives the formulas for $\pi^*$.
\end{proof}

In light of this discussion, we will often omit notation indicating what base quiver an adjunction is over, letting the context make it clear.
We use this compatibility to inductively prove some combinatorial properties of such an adjunction.  Recall that a \textbf{coatom} of a finite lattice is an element immediately preceding $\hat{1}$, that is, an element that $\hat{1}$ covers.

\begin{lemma}\label{coatomlemma}
The maps $\pi_*$ and $\pi^*$ of Proposition \ref{quivadjunctionprop} have the following properties:
\begin{enumerate}[(a)]
\item $\pi_* (\hat{1}) = M$ and $\pi^* (M) = \hat{1}$,
\item $\pi^* (N) =\pi^*(M \wedge N)$ and $\pi_* \circ \pi^* (N) = M \wedge N$,
\item they restrict to inverse bijections between the sets of coatoms of $\latt{Q_M}{\sigma}$ and $\gen{M}$.
\end{enumerate}
\end{lemma}
\begin{proof}
By definition of $\rank_M$, we have that $\pi_* (\hat{1}) = M$. Then the expression for an upper adjoint in equation (\ref{rightadjointeqn}) implies that $\pi^* (M) = \hat{1}$, and from the last statement of the same proposition we get 
\[
\pi^* (N) = \pi^* (N) \wedge \hat{1} = \pi^* (N) \wedge \pi^*(M) = \pi^*(M \wedge N) . 
\]
To see that $\pi_* \circ \pi^* (N) = M \wedge N$, we proceed by induction using Lemma \ref{adjunctioncompatlemma}, retaining the notation from the statement of the lemma.

\extcase In this case, the induction hypothesis implies that $\pi_*^i \circ \pi^*_i (T) = S_i \wedge T$ for $T \in \latt{P}{\tau}$, so the induced map on ideals sends $N\in \latt{Q}{\sigma}$ to $\gen{S_i} \wedge N$.  Then applying equations (\ref{extpusheqn}) and (\ref{extpulleqn}) gives
\[
\pi_* \circ \pi^* (N) = \bigcup_{i=1}^m \pi_*^i \circ \pi^*_i (N) = \bigcup_{i=1}^m \gen{S_i} \wedge N = M \wedge N .
\]

\gluecase By the induction hypothesis, $\pi_*^P \circ \pi^*_P (Z) = X \wedge Z$, and similarly over $S$.  So applying equations (\ref{gluepusheqn}) and (\ref{gluepulleqn}) to $N=(Z,W)$ gives
\begin{align*}
\pi_* \circ \pi^* (N) &= \pi_* \left( (\pi^*_P (Z), \pi^*_S (W) ) \right) = (\pi_*^P \circ \pi^*_P (Z), \pi_*^S \circ\pi^*_S (W) ) \\
&=   (X \wedge Z , Y \wedge W) = (X, Y) \wedge (Z, W) = M \wedge N .
\end{align*}

The third item is proven by induction.

\extcase Let $\max(M) = \{S_1, \dotsc, S_m \}$, so that the coatoms of $\gen{M}$ are the ideals
\[
I_j = M \setminus \{S_j \}
\]
obtained by removing one maximal element from $\gen{M}$.  The coatoms of $\prod_i J(\latt{P_{S_i}}{\tau_i})$ are of the form $(\hat{1}, \dotsc, \hat{1}, C_i, \hat{1}, \dotsc, \hat{1})$, where $C_i := \latt{P_{S_i}}{\tau_i} \setminus \{\hat{1}\}$ is the unique coatom of $J( \latt{P_{S_i}}{\tau_i})$. The formulas of the previous lemma, along with the previous two items of this lemma, give the desired bijection.

\gluecase  If $M = (X, Y)$, then $\gen{M} = \gen{X} \times \gen{Y}$, and the coatoms of $\gen{M}$ are of the form $(C, Y)$ and $(X, C)$ with $C$ a coatom of $\latt{P}{\sigma}$ or $\latt{S}{\sigma}$ as appropriate.
The coatoms of $\latt{Q_M}{\sigma}$ are of the form $(C', \hat{1})$ or $(\hat{1}, C')$, where $C'$ is a coatom of $\latt{P_X}{\sigma}$ or $\latt{P_Y}{\sigma}$.  If we assume that the lemma holds over $P$ and $S$, by induction, then the compatibility of $\pi^*$ and $\pi_*$ with gluing implies the lemma for $Q$.
\end{proof}

With these facts in hand, we can realize this combinatorial adjunction in a representation theoretic setting.  A technical lemma 
%gives some sufficient conditions for extending a direct sum decomposition of a representation of a rooted tree quiver $(P, \tau)$ to the extension $(Q, \sigma)$; it 
will clean up the proof of the theorem giving this realization; first we recall another notion from combinatorics which will be necessary for the lemma.
A \textbf{quasi-order} on a set $X$ is a binary relation $\preceq$ which is both
%\begin{enumerate}[]
%\item Reflexive: $x \preceq x$ for every $x \in X$, 
%\item Transitive: if $x \preceq y$ and $y \preceq z$, then $x \preceq z$.
%\end{enumerate}
reflexive and transitive.  That is, $x \preceq x$ for every $x \in X$ and if $x \preceq y$ and $y \preceq z$, then $x \preceq z$.
A quasi-order for which $x \preceq y \preceq x$ implies that $x = y$ is precisely the definition of a partial-order.  So, defining an equivalence relation on a quasi-ordered set $X$ by
\[
x \sim y \iff x \preceq y \preceq x
\]
induces a natural partial-order the set of $\sim$-equivalence classes in $X$.

\begin{lemma}
Suppose that the rooted tree quiver $(Q, \sigma)$ is obtained from $(P, \tau)$ by extension along an arrow $\tau \xto{\alpha} \sigma$, and that $V \in \repq$ is a representation of $Q$ with $\dim_K V_\sigma = 1$.  Assume that the restriction of $V$ to $P$ decomposes as $V|_P \simeq \bigoplus_{i \in I} U_i \oplus \tilde{U}$ for some subrepresentations $U_i, \tilde{U} \subset V|_P$, with $\dim_K (U_i)_\tau = 1$ and $\dim_K \tilde{U}_\tau = 0$.  Furthermore, assume that $V_\alpha$ restricts to an isomorphism
\[
V_\alpha \colon \left(U_i \right)_\tau \xto{\sim} V_\sigma
\]
on each summand $U_i$.

The set $\{ U_i \}$ is quasi-ordered by the relation
\[
U_i \preceq U_j \iff \text{there exists }f \colon U_i \to U_j \text{ such that } f_\tau \neq 0
\]
which induces an equivalence relation $\sim$ on $\{ U_i \}$ as described above.  Let $J \subseteq I$ be such that $\{ U_j \}_{j \in J}$ contains exactly one element of each maximal equivalence class, with respect to the induced partial order.

Then $V$ has a direct sum decomposition $V \simeq X \oplus Y$, with
\[
X|_P \simeq \bigoplus_{j \in J} U_j \qquad \text{and} \qquad X_\sigma = V_\sigma
\]
(and hence $Y_\sigma = 0$).
\end{lemma}
\begin{proof}
Fix a basis $\{ u_i \}$ of $V_\tau$ such that each $u_i \in U_i$, and $V_\alpha (u_i) = V_\alpha (u_j)$ for all $i, j$ (for instance, fix a nonzero $v_\sigma \in V_\sigma$ and set each $u_i$ to be the preimage of $v_\sigma$ in $U_i$).
Then for each $i \notin J$, there exists some (not necessarily unique) $j(i) \in J$ and $f_i \in \Hom_P (U_i, U_{j(i)}) $ with $(f_i)_\tau \neq 0$. 
Multiplying by a scalar, if necessary, we can assume $f_i (u_i) = u_{j(i)}$.  Now define
\[
\varphi := id - \sum_{i \notin J} f_i \in \End_P \left( \bigoplus_{i \in I} U_i \right)
\]
so that $f_i (U_j) = 0$ for all $i \notin J$ and $j \in J$.  Then $\varphi$ restricts to the identity on the subrepresentation $\bigoplus_{j \in J} U_j \subset V |_P$, hence $\varphi$ splits the inclusion of these summands.  Furthermore, $\ker \varphi \simeq \bigoplus_{i \notin J} U_i \in \rep (P)$ by the Krull-Schmidt theorem.  Since
$(\ker \varphi)_\tau$ is generated by 
\[
\setst{u_i - u_{j(i)}}{i \notin J} \subset \ker V_\alpha , 
\]
the subrepresentation $\ker \varphi$ extends by 0 to a representation $Y$ of $Q$ satisfying the conclusion of the theorem.
\end{proof}

\begin{theorem}\label{charreppullbackthm}
Let $M \in \latt{Q}{\sigma}$ and define $c:=c_M$. Let $(\pi_*, \pi^*)$ be the corresponding adjunction of Proposition \ref{quivadjunctionprop}.  Then for any $N \in \latt{Q}{\sigma}$, we have
\[
c^* \charrep_N \simeq \charrep_{\pi^* N} \oplus U
\]
where $U$ is a direct sum of reduced representations of $Q_M$ such that $\sigma \notin \supp U$.  Similarly, for $A \in \latt{Q_M}{\sigma}$ we have
\[
c_*\charrep_A \simeq \charrep_{\pi_* A} \oplus U'
\]
with $U'$ a direct sum of reduced representations of $Q$ and $\sigma \notin \supp U$.
\end{theorem}
\begin{proof}
The proof is by induction on the number of vertices of $Q$.

\extcase First consider the case that $M = \gen{S}$ is principal, so $Q_M$ is a one point extension of $(P_S, \tau)$, and there is an adjunction
\[
\xymatrix{ \latt{P_S}{\tau} \ar@<.5ex>[r]^{\pi_*}  & \ar@<.5ex>[l]^{\pi^*} \latt{P}{\tau} } . 
\]
Let $\max(N) = \{T_1, \dotsc, T_n \} \subset \latt{P}{\tau}$, so by the induction hypothesis we get
\[
c^* \charrep_N |_{P_S} = c_S^* \left( \bigoplus_{i=1}^n \charrep_{T_i} \right) \simeq \bigoplus_{i=1}^n \charrep_{\pi^*(T_i)} \oplus \tilde{U}
\]
with $\tilde{U}$ a direct sum of reduced representations of $P_S$, and $\tau \notin \supp \tilde{U}$ by dimension count.  Since $(c^* \charrep_N)_\alpha = (\charrep_N)_\alpha$ restricts to an isomorphism on each summand $\charrep_{\pi^*(T_i)}$, the representation $c^* \charrep_N$ with the decomposition above satisfies the hypothesis of the previous lemma.  By Theorem \ref{charrephomthm}, a maximal subset of $\{ \charrep_{\pi^* (T_i)} \}$ (with respect to the ordering in the previous lemma) is given by taking those representations indexed by the set of maximal elements of $\{ \pi^* (T_i) \}$.  Then the lemma gives a direct sum decomposition
\[
c^* \charrep_N \simeq W \oplus U
\]
with $U$ a direct sum of reduced representations, $U_\sigma = 0$, and $W |_P \simeq \bigoplus \charrep_{\pi^*(T_i)}$, the sum taken over maximal elements of $\{ \pi^* (T_1), \dotsc , \pi^* (T_n) \}$.  By construction of the reduced representations, $W \simeq \charrep_{\tilde{N}}$, where $\tilde{N} := \gen{\pi^*(T_1), \dotsc, \pi^*(T_n)}$.  The compatibility of $\pi^*$ with extension from Lemma \ref{adjunctioncompatlemma} then gives that $\tilde{N} = \pi^*(N)$.

For a general $M$, say with $\max(M) = \{S_1, \dotsc, S_m \}$, we have adjunctions
\[
\xymatrix{ \latt{P_{S_i}}{\tau_i} \ar@<.5ex>[r]^{\pi_*^i}  & \ar@<.5ex>[l]^{\pi^*_i} \latt{P}{\tau} }
\]
and by the case above we know that $c^* \charrep_N|_{Q_{\gen{S_i}}} \simeq \charrep_{\pi^*_i (N)} \oplus U_i$, with $U_i$ a direct sum of reduced representations.  But, since $Q = \coprod_i^\sigma Q_{\gen{S_i}}$, taking $U := \bigoplus U_i$ we get
\[
c^* \charrep_N \simeq \charrep_{(\pi^*_1 (N), \dotsc, \pi^*_m (N))} \oplus U = \charrep_{\pi^* (N)} \oplus U
\]
where the last equality follows from the compatibility of $\pi^*$ with gluing.

The proof for pushforward is similar, and $\sigma$ is not in the support of $U$ by dimension count.

\gluecase Write $M= (X, Y)$ and $N = (Z, W)$ in $\latt{Q}{\sigma} = \latt{P}{\sigma} \times \latt{S}{\sigma}$.  Over $P$, we have an adjunction
\[
\xymatrix{ \latt{P_{X}}{\sigma} \ar@<.5ex>[r]^{\pi_*} & \ar@<.5ex>[l]^{\pi^*} \latt{P}{\sigma} } . 
\]
By the induction hypothesis, we have 
\[
c^* \charrep_N |_{P_X} = c_X^* \charrep_Z \simeq \charrep_{\pi^* (Z)} \oplus U_P \qquad \text{and} \qquad c^* \charrep_N |_{S_Y} = c_Y^* \charrep_W \simeq \charrep_{\pi^* (W)} \oplus U_S . 
\]
Since $\sigma \notin \supp U_P \cup \supp U_S$ by dimension reasons, if we set $U := U_P \oplus U_S$ we get that
\[
c^* \charrep_N \simeq \charrep_{(\pi^*(Z), \pi^*(W))} \oplus U = \charrep_{\pi^*(N)} \oplus U . 
\]
The proof for pushforward is similar, and $\sigma \notin \supp U'$ by dimension reasons again.
\end{proof}

The following lemma, valid for any quiver $Q$ (not just rooted trees), gives an essential connection between the tensor product in $\repq$ and quivers over $Q$.

\begin{lemma}\label{pullpushtensorlemma}
Let $Q$ be any quiver, $c\colon Q' \to Q$ any quiver over $Q$, and $V \in \repq$.   Then there is an isomorphism
\[
\left( c_* \id_{Q'} \right) \otimes V \cong c_* c^* V . 
\]
\end{lemma}
\begin{proof}
For each $y \in Q\verts$ we have an isomorphism of vector spaces
\[
\left(c_* \id_{Q'} \otimes V\right)_y = \left(\bigoplus_{x \in c^{-1}y} \id_x \right) \otimes V_y \cong \bigoplus_{x \in c^{-1}y} V_y = \bigoplus_{x \in c^{-1}y} (c^* V)_x = \left(c_* c^* V\right)_y
\]
such that for each arrow $a \in Q\arrows$, the maps over $a$ are identified under this isomorphism:
\[
\left(c_* \id_{Q'} \otimes V \right)_a =  \sum_{b \in c^{-1}a} \id_b \otimes V_a = \sum_{b \in c^{-1}a} \id_b \otimes (c^* V)_b = \sum_{b \in c^{-1}a} (c^* V)_b = (c_* c^* V)_a . \qedhere
\]
\end{proof}

This lemma allows us to compute the tensor product of reduced representations as a corollary of the theorem.

\begin{corollary}\label{charreptensorcor}
For $M, N \in \latt{Q}{x}$, there is an isomorphism
\[
\charrep_M \otimes \charrep_N \simeq \charrep_{M \wedge N} \oplus U
\]
where $U$ is a direct sum of reduced representations without $x$ in its support.
\end{corollary}
\begin{proof}
Let $c:= c_M$, so that by Theorem \ref{charreppullbackthm} and Lemma \ref{pullpushtensorlemma}, we have
\[
\charrep_M \otimes \charrep_N \simeq c_* c^* (\charrep_N) \simeq c_* (\charrep_{\pi^*N} \oplus U') \simeq \charrep_{\pi_* (\pi^* (N))} \oplus c_*U' = \charrep_{M \wedge N} \oplus U
\]
with the last equality following from Lemma \ref{coatomlemma}. That $U$ is a direct sum of reduced representations also follows from Theorem \ref{charreppullbackthm}, and $x$ cannot be in the support of $U$ because the other three representations appearing in the formula have dimension one at $x$.
\end{proof}

For $M \in \latt{Q}{x}$, define the corresponding \textbf{rank function} on $V \in \repq$ by
\[
r_M (V) := \dim_K (\rank_M V) \in \Z_{\geq 0} . 
\]

\begin{corollary}\label{zetacor}
Evaluation of rank functions on reduced representations is given by the zeta function of $\latt{Q}{x}$.  That is, for $M, N \in \latt{Q}{x}$, we have
\[
r_M (\charrep_N) = \zeta (M, N) :=
\begin{cases}
1 & M \leq N \\
0 & \text{otherwise} \\
\end{cases} . 
\]
\end{corollary}
\begin{proof}
From Theorem \ref{kinserrankthm}, $r_M (\charrep_N)$ counts the number of direct summands of $c_M^*\charrep_N$ which are isomorphic to $\id_{Q_M}$.  By Theorem \ref{charreppullbackthm}, this number is nonzero if and only if $\pi^*(N) = \hat{1}$, which is if and only if $M \leq N$ by Lemma \ref{coatomlemma}.  This number can be at most 1 since $\dim_K (\charrep_N)_\sigma = 1$.
\end{proof}

This corollary strengthens part (a) of Proposition \ref{rankspaceorderprop} to: ``$N \geq M$ in $\latt{Q}{x}$ if and only if there is an inclusion of functors $\rank_N \subseteq \rank_M$''.

\begin{remark}
With a little more work, the previous corollary can be extended to show that the relation over $\latt{Q}{}$ given by
\[
M \leq N \iff r_M (\charrep_N) \neq 0
\]
for all $M, N \in \latt{Q}{}$ is a partial ordering.  This partial order on all of $\latt{Q}{}$ can also be obtained in a purely combinatorial way, without reference to rank functors or representations, by ``patching together'' the lattices $\latt{Q}{x}$. 
More precisely, we start with the partial order on $\latt{Q}{}$ inherited from its definition as the disjoint union of all $\latt{Q}{x}$, and add certain relations for each arrow of $Q$ as follows.  Let $a$ be an arrow of $Q$.  By construction, we have an inclusion
\[
J(\latt{Q}{ta}) \subseteq \latt{Q}{ha} \subseteq \latt{Q}{} .
\]
Then for every $S \in \latt{Q}{ta}$, we add the relation $S \leq \gen{S}$ in $\latt{Q}{}$, and refine the inherited partial order on $\latt{Q}{}$ to include these additional relations.  Then the cover relations in this partial ordering of $L_Q$ correspond to more morphisms between rank functors and between reduced representations.  However, $L_Q$ is not a lattice if $Q$ has more than one vertex, and most of our correspondences between combinatorial properties of $L_Q$ and representation theoretic statements do not generalize as neatly as Corollary \ref{zetacor}.  It is much easier to work with the individual lattices $\latt{Q}{x}$, and this ordering on the entire set $L_Q$ will not be needed in any proofs in this paper; hence we omit formal proofs of the statements in this remark.
\end{remark}

%%%%%%%%%%%%%%%%%%%%%%%%%%%%%%%%%%%%%%%%%%%%%%%
%				REP RINGS OF ROOTED TREE QUIVERS						%
%%%%%%%%%%%%%%%%%%%%%%%%%%%%%%%%%%%%%%%%%%%%%%%
\section{The Representation Ring}
\label{repringsect}
We now have enough tools to start an analysis of the representation ring of a rooted tree quiver.  The \textbf{representation ring} $R(Q)$ of a quiver $Q$ (cf. \cite[\S3]{kinserrank}) is defined as the free abelian group on the isomorphism classes of representations of $Q$, modulo the subgroup generated by elements $[V] + [W] - [V \oplus W]$.  By the Krull-Schmidt theorem \cite[Thm.~I.4.10]{assemetal}, $R(Q)$ is freely generated by indecomposable representations.  Multiplication is given by $[V][W] = [V \otimes W]$, so $R(Q)$ is commutative with identity $[\id_Q]$.

%%%%%%%%%%%%%%%%%%%%%%%%%%%%%%%%%%%%%%%%%%%%%%%
%						SUPPORT SUBALGEBRA						%
%%%%%%%%%%%%%%%%%%%%%%%%%%%%%%%%%%%%%%%%%%%%%%%
\subsection{The Support Algebra of a Quiver}
\label{supportsect}
In this subsection, $Q$ can be any quiver, possibly disconnected or with oriented cycles and parallel arrows.  The set $\subquivs$ of connected subquivers of $Q$ is partially ordered by inclusion.  We will show that  the M\"obius algebra of $\subquivs$ over $\Z$ is naturally a subalgebra of the representation ring of $Q$, decomposing $R(Q)$ into a product of rings.  This generalizes some results of \cite[\S 3]{herschend07b}, while at the same time putting them in a natural combinatorial setting.

The \textbf{M\"obius algebra} (over $\Z$) of a poset $P$ (cf. \cite{greenemobiusalgebra}), written $A(P, \Z)$, is defined as the free $\Z$-module on the elements of $P$, with multiplication of two of these basis elements $x, y \in P$ given by
\begin{equation}\label{mobiusmulteqn}
x \cdot y := \sum_{s \in P} \left[ \sum_{s \leq t \leq x, y} \mu (s, t) \right] s
\end{equation}
where $\mu$ is the M\"obius function of $P$.  In case the meet of $x$ and $y$ exists, this simplifies to $x\cdot y = x \wedge y$, so if $P$ is a meet semi-lattice then $A(P, \Z)$ is just $\Z[L; \wedge]$, the semi-group ring of $L$ with respect to the meet operator.  The multiplication defined in (\ref{mobiusmulteqn}) is precisely the structure that gives a $\Z$-basis for $A(P, \Z)$ of orthogonal idempotents
\[
\delta_x := \sum_{y \leq x} \mu(y, x) y \qquad x \in P
\]
via M\"obius inversion.  The original basis elements can be recovered as $y = \sum_{x \leq y} \delta_x$.

The poset $\subquivs$ is a meet semi-lattice if and only if $Q$ is a tree.
When $Q$ is not a tree, we can still view the multiplication in $A:=A(\subquivs, \Z)$ in a more intuitive way than one might expect from the expression in (\ref{mobiusmulteqn}).
Given $P, S \in \subquivs$, let $\{M_i\}$ be the set of maximal connected subquivers (i.e., the connected components) of $P \cap S$.  If $U$ is connected and contained in both $P$ and $S$, then $U$ is contained in some \emph{unique} $M_i$.  Furthermore, any connected $T$ containing $U$ but also contained in both $P$ and $S$ will itself also be contained in $M_i$.  In other words, for each $U \leq P, S$ there exists a unique $i$ such that
\[
\setst{T \in \subquivs}{U \leq T \leq P, S} = \setst{T \in \subquivs}{U \leq T \leq M_i} . 
\]
So the bracketed sum in equation (\ref{mobiusmulteqn}) can be computed for $P, S \in \subquivs$ to be
\[
\sum_{U \leq T \leq P, S} \mu (U, T) = \sum_{U \leq T \leq M_i} \mu (U, T) = \begin{cases}
1 & U = M_i \\
0 & U \neq M_i
\end{cases}
\]
using the standard property of the M\"obius function (cf. \cite[\S~3.7]{stanleyenumcombin}).  Hence, the product of two connected subquivers in $A$ is the sum of the connected components of their intersection $P \cap S$:
\[
P \cdot S = \sum_i M_i . 
\]

There is an injective map $\phi \colon \subquivs \to R(Q)$ given by $\phi(P) = \id_P$.  Since each $P \in \subquivs$ is connected, $\id_P$ is indecomposable, so the image of $\phi$ is a set of $\Z$-linearly independent elements of $R(Q)$.  Hence $\phi$ uniquely extends to map of $\Z$-modules $\tilde{\phi} \colon A \into R(Q)$.
From the definition of quiver tensor product, it is easy to see that
\[
\id_P \otimes \id_S = \bigoplus_i \id_{M_i}
\]
where again $\{ M_i \}$ is the set of connected components of $P \cap S$.  This shows that $\tilde{\phi}$ is a ring homomorphism, so we can regard $A$ as a subalgebra of $R(Q)$ by identifying a connected subquiver of $Q$ with the identity representation of that subquiver.

\begin{definition}
We call $A(\subquivs, \Z)$ (or its natural image in $R(Q)$) the \textbf{support algebra} of $Q$.
\end{definition}

In particular, we can define for each $P \in \subquivs$ an element
\[
e_P := \sum_{P' \leq P} \mu(P', P) \id_{P'}
\]
of the support algebra such that $\id_P = \sum_{P' \leq P} e_{P'}$, and $\{e_P\}_{P \in \subquivs}$ is a set of orthogonal idempotents in $R(Q)$.  From this discussion, and the fact that orthogonal idempotents give a direct product decomposition of a ring, the following proposition is immediate.
\begin{prop}
For any quiver $Q$, the support algebra gives a decomposition of $R(Q)$ into a product of rings
\[
R(Q) \cong \prod_{P \in \subquivs} e_P R(Q)
\]
called the \textbf{decomposition of $R(Q)$ by supports}.
\end{prop}

The next proposition gives a first entry in the dictionary between $R(Q)$ and $\repq$.

\begin{prop}\label{eppropertiesprop}
For a connected subquiver $Q' \subseteq Q$, the following hold:
\begin{enumerate}[(a)]
\item For any $V \in \repq$, we have that $Q' \nsubseteq \supp V$ implies $e_{Q'} V = 0$. 
\item The set of images $\{e_{Q'} V \}$ freely generates $e_{Q'} R(Q)$ as a $\Z$-module, where $V$ ranges over isomorphism classes of indecomposable representations with support exactly $Q'$.  In particular, if $V$ is indecomposable and $\supp V = Q'$, then $e_{Q'} V \neq 0$.
\end{enumerate}
\end{prop}

\begin{proof}
\begin{enumerate}[(a)]

\item Let $P:=\supp V$, so we have that $V = \id_P V = \sum_{P' \leq P} e_{P'} V$.  Then
\[
e_{Q'}V = e_{Q'} \id_P V =
\begin{cases}
e_{Q'} V 	& Q' \subseteq P \\
0		& Q' \nsubseteq P
\end{cases}
\]
by orthogonality.

\item This part reduces to the case $Q'=Q$ by induction on the number of vertices of $Q$.

Since $R(Q)$ is generated as a $\Z$-module by indecomposable representations, the images of the indecomposables generate the factor ring $e_Q R(Q)$.  But if $\supp V \nsupseteq Q$ (i.e. $\supp V \neq Q$ since $Q$ is the maximal subquiver of itself), then $e_Q V = 0$ by (a), so in fact $e_Q R(Q)$ is generated by the images of indecomposables with support exactly $Q$.

Now suppose there is a relation
\[
\sum_i n_i e_Q V_i = 0
\]
where $\{V_i \}$ are pairwise non-isomorphic indecomposables with $\supp V_i =Q$ and $n_i \in \Z$.  
Then substituting the expression
\[
e_Q = \id_Q + \sum_{P < Q} \mu(P,Q) \id_P
\]
and using that $\id_Q V_i = V_i$ for all $i$, we would get
\[
\sum_i n_i V_i = - \sum_i \sum_{P < Q} n_i \mu(P, Q) \id_P V_i . 
\]
Every term $\id_P V_i$ on the right hand side has support smaller than $Q$, and the terms on the left hand side are indecomposable with support $Q$.  Since indecomposables freely generate $R(Q)$ by definition, it must be that each $n_i = 0$.  \qedhere
\end{enumerate}
\end{proof}

It should be noted again that idea of giving an orthogonal decomposition of $R(Q)$ in order to simplify inductive proofs is due to Herschend, and that some parts of the above propositions appear in the work cited above, under additional assumptions (e.g. $Q$ is a tree, $Q$ is Dynkin).

%%%%%%%%%%%%%%%%%%%%%%%%%%%%%%%%%%%%%%%%%%%%%%%
%								FINE SUPPORT							%
%%%%%%%%%%%%%%%%%%%%%%%%%%%%%%%%%%%%%%%%%%%%%%%
\subsection{A finer notion of support}
\label{fsuppsect}
Now suppose that $(Q, \sigma)$ is a rooted tree quiver, and let $R:=R(Q)$ be the representation ring of $Q$.  We will use the reduced representations $\charrep_M$ and lattices $\latt{Q}{x}$ to further decompose each factor $e_P R$.  The main result of the paper will be that, after this, no further decomposition is possible.  More precisely, we will show that $L_Q$ indexes a complete set of orthogonal idempotents in $R$.

In the remainder of the paper, we will simplify some of the constructions and proofs by ignoring direct summands of representations which don't have $\sigma$ in their support, assuming that we know about these summands by some induction.  Working in an appropriate factor ring of $R$ is the technical tool that allows us to do this rigorously.  Writing $\subquivs_\sigma := \setst{P \subseteq Q}{\sigma \in P\verts}$ for the collection of all connected subquivers of $Q$ with $\sigma$ in their support, we define
\[
R_\sigma = \prod_{P \in \subquivs_\sigma} e_P R . 
\]
Then $R_\sigma $ is naturally both an ideal in $R$, and a factor ring of $R$ with identity $\sum_{P \in \subquivs_\sigma} e_P$.  For $r \in R$, we denote by 
\[
\overline{r} := \left( \sum_{P \in \subquivs_\sigma} e_P \right) r 
\]
the image of $r$ in $R_\sigma$.  By Proposition \ref{eppropertiesprop}, $R_\sigma$ is freely generated as a $\Z$-module group by the images of all indecomposable representations of $Q$ with $\sigma$ in their support.
%\[
%\ind_\sigma Q := \setst{\overline{V} \in R_\sigma}{V \in \repq \text{ indecomposable and } \sigma \in \supp V}
%\]
The following lemma justifies why induction reduces the study of $R$ to that of $R_\sigma$, and  interprets passage to $R_\sigma$ in terms of the representation theory of $Q$.

\begin{lemma}
Let $X \subset Q\verts$ be the set of vertices $x$ of $Q$ for which there exists an arrow from $x$ to $\sigma$.  Then
\[
R \cong R_\sigma \times \prod_{x \in X} R(Q\tox) %\tox
\]
and for $V \in \repq$, we have $\overline{V} =0$ in $R_\sigma$ if and only if $\sigma \notin \supp V$.
\end{lemma}
\begin{proof}
The first statement holds because if $\sigma \notin \supp P$ for some connected $P \subset Q$, then $P \subseteq Q\tox$ for a unique $x$.  The second statement is a corollary of Proposition \ref{eppropertiesprop}.
\end{proof}

Now consider the M\"obius algebra of $\latt{Q}{\sigma}$ over $\Z$, which we will denote by
\[
A_\sigma := A(\latt{Q}{\sigma}, \Z) . 
\]
Since $\latt{Q}{\sigma}$ is a lattice, this is the semigroup algebra of $\latt{Q}{\sigma}$ with respect to the meet operator $\wedge$.
%, that is, $M \cdot N = M \wedge N$ for $M, N \in \latt{Q}{\sigma}$.  
The map $\psi \colon \latt{Q}{\sigma} \to R_\sigma$ given by $\psi(M) = \overline{\charrep_M}$ extends by linearity to a map of $\Z$-modules
\[
\tilde{\psi} \colon A_\sigma \into R_\sigma . 
\]
Proposition \ref{eppropertiesprop} implies that this map is injective, since the reduced representations $\charrep_M$ for $M \in \latt{Q}{\sigma}$ are indecomposable and pairwise non-isomorphic by Corollary \ref{indecompcor}.
Using Corollary \ref{charreptensorcor}, in $R_\sigma$ we have
\[
\overline{\charrep_M} \overline{\charrep_N} = \overline{\charrep_{M \wedge N}}
\]
so in fact $\psi$ is a ring homomorphism.  Thus $A_\sigma$ is a subalgebra of $R_\sigma$, and as before we get orthogonal idempotents
\[
f_M := \sum_{M' \leq M} \mu(M', M) \overline{\charrep_{M'}} . 
\]
 in $R_\sigma$ which give a direct product decomposition
\[
R_\sigma \cong \prod_{M \in \latt{Q}{\sigma}} f_M R_\sigma . 
\]
In particular, note that
\begin{equation}\label{charreptimesfeqn}
f_M \charrep_N = f_M \overline{\charrep_N} = f_M\, \zeta(M, N)
\end{equation}
where $\zeta$ is again the zeta function of the poset $\latt{Q}{\sigma}$, as in Corollary \ref{zetacor}.  This simply follows from substituting the expression $\overline{\charrep_N} = \sum_{N' \leq N} f_{N'}$.

In Propostion \ref{eppropertiesprop}, we related the images of a representation in the factor rings $e_P R$ to a basic representation theoretic property, namely the support of a representation.  Our goal now is to add another entry to the dictionary between $R(Q)$ and $\repq$ by doing something analogous for the factor rings $f_M R_\sigma$.  %However, since we start with the factor ring $R_\sigma$, we won't get a property of representations, but rather of equivalence classes of representations, up to direct summands without $\sigma$ in their support.

\begin{definition}
For $V \in \repq$ such that $\overline{V} \neq 0$ in $R_\sigma$, let
\[
\fsuppset_V = \setst{ M \in \latt{Q}{\sigma}}{ \overline{\charrep_M V} = \overline{V}} . 
\]
We define the \textbf{fine support} of $V$ to be
\[
\fsupp V := \min \fsuppset_V \in \latt{Q}{\sigma} . 
\]
\end{definition}

The set has a unique minimal element because $\fsuppset_V$ is a meet semi-lattice of $\latt{Q}{\sigma}$.  That is, if both $\overline{\charrep_M V }= \overline{V}$ and $\overline{\charrep_N V} = \overline{ V}$, then 
\[
\overline{\charrep_{M \wedge N} V}  = \overline{\charrep_M \charrep_N V} = \overline{ V}.
\]
Furthermore, every $M \geq \fsupp V =:N$ has the property that $\overline{\charrep_M V} = \overline{V}$, since
\[
\overline{\charrep_M V} = \overline{\charrep_M \left(\charrep_{N} V \right) } = \overline{\left(\charrep_M \charrep_N \right) V} = \overline{\charrep_{M \wedge N} V} = \overline{\charrep_{N} V} = \overline{V} . 
\]
In other words, $\fsuppset_V$ is always a principal filter (or dual order ideal) in $\latt{Q}{\sigma}$, and by definition $\fsupp V$ is its generator.  
Now suppose that $\overline{W} \neq 0$ for every indecomposable summand $W$ of $V$.  Then by considering the dimension at $\sigma$, an alternative characterization of fine support is that $N \geq \fsupp V$ if and only if $\charrep_N \otimes V \simeq V \oplus U$ for some $U \in \repq$.
The next proposition gives the basic properties of fine support. Note that the first four properties are 
analogous to properties of the support of a representation.

\begin{prop}  In the statements below, assume that every representation appearing has nonzero image in $R_\sigma$, so the fine support is defined.  Then the following properties hold: 
\begin{enumerate}[({F}1)]
\item For a direct sum decomposition $V \simeq \bigoplus_{i=1}^n V_i$, we get
\[
\fsupp V = \bigvee_{i=1}^n \fsupp V_i .
\]

\item Tensor product can only decrease the fine support of a representation.  More precisely, we have
\[
\fsupp (V \otimes W) \leq \fsupp V \wedge \fsupp W .
\]

\item If $\fsupp V \ngeq M$, then $f_M V = 0$.

\item The ring $f_M R$ is freely generated as a $\Z$-module by
\[
\setst{f_M V}{V \text{ indecomposable with } \fsupp V = M}.
\]
In particular, if $V$ is indecomposable and $\fsupp V = M$, then $f_M V \neq 0$.

\item Whenever $\rank_M V \neq 0$,  we have that $M \leq \fsupp V$.

\item Suppose $V$ is indecomposable and $\fsupp V = N$. Then $\rank_N V \neq 0$ implies $V \simeq \charrep_N$.
\end{enumerate}
\end{prop}

%\begin{remark}
%In the analogues of properties (F2) and (F3) for support, ``$\leq$'' and ``implies that'' can be replaced with ``$=$'' and ``if and only if'', respectively.  These statements are sharp for fine support, however.
%\end{remark}
\begin{proof}  Let $N := \fsupp V$ throughout the proof.

\begin{enumerate}[({F}1)]
\item Let $M :=\bigvee_{i=1}^n \fsupp V_i$, so we want so show $M=N$.  Since $M \geq \fsupp V_i$ for each $i$, we know that $\overline{\charrep_M V_i} = \overline{V_i}$ for each $i$.  Thus we can compute
\[
\overline{\charrep_M V} = \overline{\charrep_M \sum_i V_i} = \sum_i \overline{\charrep_M V_i} = \sum_i \overline{V_i} = \overline{V}
\]
and so  $N \leq M$.

For the reverse inequality, it is enough to show that $N \geq \fsupp V_i$ for each $i$, and without loss of generality we can take each $V_i$ to be indecomposable.
%, that is, that $\overline{\charrep_N V_i} = \overline{V_i}$.
  Let $\{W_{ij} \}$ be the set of indecomposable summands of $\charrep_N \otimes V_i$ that have $\sigma$ in their support, so $\overline{ \charrep_N V_i} = \sum_{j=1}^{d(i)} \overline{W_{ij}}$.  Then by assumption, we get an equality
\[
\sum_i \overline{V_i} = \overline{V} = \overline{\charrep_N V} = \overline{\charrep_N} \sum_i \overline{V_i} = \sum_i \overline{\charrep_N V_i} = \sum_{ij} \overline{W_{ij}} . 
\]
Since the $V_i$ and $W_{ij}$ are indecomposable, Proposition \ref{eppropertiesprop} implies that each $d(i) = 1$ and there is a permutation $\pi \in \symgp{n}$ such that $\overline{ \charrep_N V_i} = \overline{ V_{\pi i}}$ for all $i$.  But since $\overline{\charrep_N}$ idempotent in $R_\sigma$, the associated permutation $\pi$ is also, and so $\pi$ is the identity permutation.  Thus $\overline{ \charrep_N V_i} = \overline{ V_{i}}$ for all $i$, which implies that $N \geq \fsupp V_i$ for all $i$, and finally that $N \geq M$.

\item If we write $M:= \fsupp W$, then we can use Corollary \ref{charreptensorcor} to compute
\[
\overline{\charrep_{N \wedge M}}  (\overline{V \otimes W}) = \overline{\charrep_N \charrep_M VW} = \overline{\charrep_N V}\ \overline{\charrep_M W} = \overline{V}\overline{W} . 
\]
So by definition we get that $\fsupp(V \otimes W) \leq N \wedge W$.

\item If $N \ngeq M$, then equation (\ref{charreptimesfeqn}) implies that $f_M V = f_M (\overline{\charrep_N V}) = 0$.

\item The ring $R_\sigma$ is generated as a $\Z$-module by the images of indecomposable representations with $\sigma$ in their support, and $f_M R$ is a factor ring of $R_\sigma$, so the image of this set generates $f_M R$ also.
For $V$ indecomposable with $\sigma \in \supp V$, we can write
\[
\overline{\charrep_M V} = \sum_{i \in I} \overline{V_i} \neq 0
\]
where each $V_i$ is indecomposable and has $\sigma$ in its support.  Then for each $i$, it must be that $\fsupp V_i \leq M$, since properties (F1) and (F2) give that 
\[
\fsupp V_i \leq \fsupp \left( \bigoplus_i V_i \right) = \fsupp \left( \charrep_M \otimes V \right) \leq M \wedge \fsupp V \leq M .
\]
Now (F3) implies that $f_M V_i = 0$ when $\fsupp V_i$ is strictly less than $M$, so we can use equation (\ref{charreptimesfeqn}) to get
\[
f_M V = (f_M \charrep_M) V = f_M (\charrep_M V) = f_M \sum_{i \in J} V_i
\]
where $J := \setst{i \in I}{\fsupp V_i = M}$.  % is the subset of indices corresponding to summands with fine support exactly equal to $M$.
Hence $f_M R$ is generated by images of indecomposables with fine support exactly $M$.

Suppose we had a relation of the form
\[
\sum_i n_i f_M V_i = 0 \qquad n_i \in \Z
\]
where each $V_i \in \repq$ has fine support exactly $M$, and the $V_i$ are pairwise non-isomorphic. % \not \simeq V_j$ when $i \neq j$.
Then substituting the expression 
\[
f_M = \overline{\charrep_M} + \sum_{M' < M} \mu(M', M) \overline{\charrep_{M'}}
\]
into the previous equation, we use the assumption that $\overline{\charrep_M V_i} = \overline{V_i}$ for each $i$ to get
\[
\sum_i n_i \overline{V_i} + \sum_i \sum_{M' < M} n_i \mu(M', M) \overline{\charrep_{M'} V_i} = 0 . 
\]
Now the double sum lies in the span of images of indecomposables with fine support strictly less than $M$, by property (F2), and the first sum consists of images of indecomposables with fine support exactly $M$.  Since the images of indecomposables with $\sigma$ in their support freely generate $R_\sigma$, we get that $n_i =0$ for all $i$.  This shows that $f_M R$ is \emph{freely} generated by the images of the indecomposables with fine support $M$.

\item Recalling that we defined  $N=\fsupp V$, we have that $\overline{\charrep_N V} = \overline{V}$ holds by definition, and so there exist $X, Y \in \repq$ with $\sigma$ not in either of their supports
%\notin \supp X \cup \supp Y$
and such that
\[
X \oplus \left( \charrep_M \otimes V \right) \simeq V \oplus Y . 
\]
Since $\sigma$ is not in the support of $X$, it must be that $\rank_M X \subseteq X_\sigma = 0$, and similarly $\rank_M Y = 0$.  Hence we can compute
\begin{multline*}
(\rank_M \charrep_N) \otimes_K (\rank_M V) = \rank_M (\charrep_N \otimes V) = \rank_M (X \oplus (\charrep_N \otimes V) ) \\ = \rank_M (V \oplus Y) = \rank_M V \neq 0
\end{multline*}
which implies that $\rank_M \charrep_N \neq 0$.  Then by Corollary \ref{zetacor} we get that $M \leq N = \fsupp V$.

\item Let $c := c_N$.  If $\rank_N V = \rank_{Q_N} (c^* V) \neq 0$,
then by Theorem \ref{kinserrankthm} there is a decomposition $c^* V \simeq \id_{Q_N} \oplus U$ for some $U \in \rep (Q_N)$. Pushing back down to $Q$ we get
\[
c_{*} c^* V \simeq c_* \id_{Q_N} \oplus c_* U = \charrep_N \oplus c_* U . 
\]
By Lemma \ref{pullpushtensorlemma}, the left hand side is isomorphic to $\charrep_N \otimes V$.  If $V$ is indecomposable and $\fsupp V = N$, then $V$ itself is the only indecomposable direct summand of $\charrep_N \otimes V$ with $\sigma$ in its support, by considering dimension at the vertex $\sigma$.  By comparison with the right hand side, it must be that $V \simeq \charrep_N$. \qedhere
\end{enumerate}
\end{proof}

%\begin{remark}
%Although the analogues of (F2) and (F3) for support are stronger (equality and equivalence, respectively), these statements about fine support cannot be strengthened in general.
%\end{remark}

The rank spaces and fine support of a representation $V$ give information about morphisms between reduced representations and $V$.

\begin{lemma}\label{rankhomlemma}
For any $M \in \latt{Q}{\sigma}$, the vectors in $\rank_M V \subseteq V_\sigma$ are precisely the vectors contained in the image at $\sigma$ of some morphism from $\charrep_M$ to $V$.  That is, there is a natural map of vector spaces
\[
\Psi\colon \Hom_Q (\charrep_M, V) \onto \rank_M V
\]
such that if we fix a nonzero vector $v_\sigma \in (\charrep_M)_\sigma$, then we have
\[
\Psi (f) = f(v_\sigma)
\]
for $f \colon \charrep_M \to V$.
\end{lemma}
\begin{proof}
Let $c:= c_M$.  Then $(c_{*}, c^*)$ is an adjoint pair by Proposition \ref{adjointprop}, so there is an isomorphism
\[
\Hom_Q (\charrep_M, V) = \Hom_Q (c_* \id , V) \cong \Hom_{Q_M} (\id, c^* V) . 
\]
From \cite[Prop.~28,~29]{kinserrank} there is a surjective linear map
\[
\Hom_{Q_M} (\id, c^* V) \onto \surjrep_{Q_M} (c^*V)_\sigma
\]
sending a morphism $f$ to the vector $f(v_\sigma)$.  Using Lemma \ref{rootedtreeranklemma}, the right hand side is equal to $\rank_M V$.
\end{proof}

\begin{lemma}\label{homtocharreplemma}
If $V \in \repq$ is such that $\fsupp V \leq M$, then any linear functional on the vector space $V_\sigma$ lifts to a morphism of representations from $V$ to $\charrep_M$.  More precisely, if we fix an isomorphism $(\charrep_M)_\sigma \simeq K$, then there is a natural map of vector spaces
\[
\Phi \colon \Hom_K( V_\sigma, K) \into \Hom_Q (V, \charrep_M)
\]
such that for $g\colon V_\sigma \to K$,
\[
\Phi(g)_\sigma = g . 
\]
\end{lemma}
\begin{proof}
Let  $c := c_M$.  Then $c^* \charrep_M \simeq \id_{Q_M} \oplus U$ with $\sigma \notin \supp U$ by Theorem \ref{charreppullbackthm}.  There is a natural isomorphism of vector spaces
\[
\Hom_K ((c^*V)_\sigma, K) \cong \Hom_{Q_M} (c^* V, \id_{Q_M})
\]
coming from the fact that $\id_{Q_M}$ is the injective representation of $Q_M$  associated to the vertex $\sigma$, (cf. \cite[Lem.~III.2.11]{assemetal}).
Now since $V_\sigma = (c^*V)_\sigma$, we use the decomposition of $c^* \charrep_M$ above to get an embedding
\[
\Hom_K (V_\sigma, K) \into \Hom_{Q_M} (c^* V, c^* \charrep_M).
\]
Since $\fsupp V \leq M$, there is a decomposition $c_* c^* V \simeq V \oplus W$ with  $\sigma \notin \supp W$.  The adjoint pair $(c_*, c^*)$ gives an isomorphism
\[
\Hom_{Q_M} (c^* V, c^* \charrep_M) \cong \Hom_Q (c_* c^* V, \charrep_M) = \Hom_Q (V, \charrep_M) \oplus \Hom_Q (W, \charrep_M) . 
\]
which, projected to the first summand on the right hand side, gives $\Phi$ as stated.
\end{proof}

In general, suppose we have a quiver representation $V$, and that we know the isomorphism class of $V|_P$ for every proper subquiver $P \subset Q$. Then we cannot deduce the isomorphism class of $V$ in $\repq$ without further information regarding how to glue together the restricted representations. In our case, we are essentially facing this problem when we try to inductively study representations of $Q$ via gluing and extension of rooted tree quivers.
The base change lemma below is a tool that allows us to utilize information in the rank functors to address this problem.

We will introduce some new notation.  If $U \subseteq V$ are representations of $Q$, and $Z \subseteq V_\sigma$ a vector subspace, then the notation 
\[
U \simeq \charrep_M \otimes_K Z
\]
 means that $U \simeq \charrep_M^{\oplus \dim_K Z}$ and $U_\sigma = Z$.  This does not uniquely identify $U$ as a subrepresentation of $V$. In this case, for any $W \in \repq$ there is an isomorphism
\begin{equation}\label{charrephomtensoreqn}
%\begin{split}
\Hom_Q (W, \charrep_M ) \otimes_K U_\sigma \cong \Hom_Q (W, \charrep_M \otimes_K U_\sigma) =\Hom_Q (W, U) %\\
%\sum_i f_i \otimes u_i &\mapsto \left( w \mapsto \sum_i f_i (w) \otimes u_i \right) .
%\end{split}
\end{equation}
given by sending an indecomposable tensor $f \otimes u$ on the right hand side to the morphism $w \mapsto f(w) \otimes u$.  It is easy to see that the map is injective, so isomorphism follows from the fact that both spaces have the same dimension.
Similarly, we get that
\begin{equation}\label{charrephomtensor2eqn}
\Hom_Q (U, W) \cong \Hom_Q (\charrep_M, W ) \otimes_K U_\sigma^* . 
\end{equation}

Now we can state and prove the base change lemma that will be our key to getting gluing data from the rank functors.

\begin{lemma}\label{basechangelemma}
Suppose $U, W \subset V$ are subrepresentations such that $V \simeq U \oplus W$ and $U \simeq \charrep_M \otimes_K U_\sigma$ for some $M \in \latt{Q}{\sigma}$. 
\begin{enumerate}[(a)]
\item Suppose $\fsupp W \leq M$, and $\tilde{Z} \subseteq V_\sigma$ is a subspace such that $V_\sigma = U_\sigma \oplus \tilde{Z}$. Then there exists a subrepresentation $Z \subset V$ such that $Z_\sigma = \tilde{Z}$ and $V \simeq U \oplus Z$.
\item Now suppose $W$ is arbitrary, and $\tilde{Z} \subseteq \rank_M V\subset V_\sigma$ is such that $V_\sigma = \tilde{Z} \oplus W_\sigma$.  Then there exists a subrepresentation $Z \subset V$ such that $Z_\sigma = \tilde{Z}$ and $V \simeq Z \oplus W$.
\end{enumerate}
\end{lemma}
\begin{proof}
\begin{enumerate}[(a)]
\item There is a short exact sequence of vector spaces
\[
0 \to \tilde{Z} \xto{i} {U_\sigma \oplus W_\sigma} \xto{(\tilde{f} \ \tilde{g})} {U_\sigma} \to 0
\]
where $i$ is the subspace inclusion, and $\tilde{f}$ is invertible since $V_\sigma = U_\sigma \oplus \tilde{Z}$.  By composing the last map with $\tilde{f}^{-1}$, and replacing $\tilde{g}$ with $\tilde{f}^{-1} \circ \tilde{g}$, we can assume $\tilde{f} = id$ without loss of generality.
By Lemma \ref{homtocharreplemma} and the isomorphism (\ref{charrephomtensoreqn}), we have an injective linear map
\[
\Hom_K (W_\sigma, U_\sigma ) \cong W_{\sigma}^* \otimes_K U_\sigma \into \Hom_Q (W, \charrep_M) \otimes_K U_\sigma \cong \Hom_Q (W, U)
\]
which gives a morphism $g \in \Hom_Q (W, U)$ such that $g_\sigma = \tilde{g}$.  Thus we can define $Z \subset V$ by the split exact sequence in $\repq$
\[
0 \to Z \to U \oplus W \xto{(id \ g)} U \to 0 .
\]

\item Similarly, we can take a short exact sequences of vectors spaces of the form
\[
0 \to \tilde{Z} \xto{i} {U_\sigma \oplus W_\sigma} \xto{(\tilde{f} \ -id)} {W_\sigma} \to 0
\]
and again we want to lift $\tilde{f}$ to some $f \in \Hom_Q (U, W)$.  Let $\pi_U, \pi_W$ be the projections given by the decomposition $V \simeq U \oplus W$.  Since $\tilde{Z} \subseteq \rank_M V$ and $\rank_M$ is additive, we get that 
\[
\pi_W (\tilde{Z}) \subseteq \pi_W (\rank_M V) = \rank_M W .
\]
Now the projections give a decomposition $\tilde{Z} = \pi_U (\tilde{Z}) \oplus \pi_W (\tilde{Z})$, and then exactness of the sequence implies that $(\tilde{f} \circ \pi_U - \pi_W) (\tilde{Z}) = 0$, or 
\[
\tilde{f} \circ \pi_U (\tilde{Z}) = \pi_W (\tilde{Z}) . 
\]
Since $W_\sigma$ is the kernel of the projector $\pi_U$ restricted to $V_\sigma$, and $\tilde{Z} \cap W_\sigma = 0$, we find that $\pi_U (\tilde{Z}) = U_\sigma$ and so $\tilde{f} \in \Hom_Q (U_\sigma, \rank_M W)$.
By Lemma \ref{rankhomlemma}, we have a surjection $\Hom_Q  (\charrep_M, W) \onto \rank_M W$ that we can tensor with $U_\sigma^*$ to get
\[
\Hom_Q (U, W) \cong \Hom_Q (\charrep_M, W) \otimes_K U_\sigma^*\onto \rank_M W \otimes_K U_\sigma^* \cong \Hom_K (U_\sigma , \rank_M W)
\]
using (\ref{charrephomtensor2eqn}).  Thus we can lift $\tilde{f}$ to some $f \in \Hom_Q (U, W)$ such that $f_\sigma = \tilde{f}$, and we can again define $Z \subset V$ by the split exact sequence of representations
\[
0 \to Z \to U \oplus W \xto{(f \ -id)} W \to 0 . \qedhere
\]
%in $\repq$.
\end{enumerate}
\end{proof}

%%%%%%%%%%%%%%%%%%%%%%%%%%%%%%%%%%%%%%%%%%%%%%%
%				 STATEMENT OF MAIN RESULTS							%
%%%%%%%%%%%%%%%%%%%%%%%%%%%%%%%%%%%%%%%%%%%%%%%
\section{Structure of the Representation Ring of a Rooted Tree}
\label{mainresultsect}

\subsection{Statements of the structure theorems}
The main result of this paper has two equivalent formulations: firstly, as a property of the representation ring of a rooted tree quiver, and secondly, as a ``splitting principal" for representations of such a quiver.
Since rank functors commute with tensor product, we have an equality of vector spaces $\rank_M (V^{\otimes n}) = (\rank_M V )^{\otimes n}$ inside $V^{\otimes n}_x$ for any $M \in \latt{Q}{x}$ and $n \in \Z_{\geq 1}$.  Consequently, we omit the parentheses in this situation. 

\begin{theorem}\label{mainthm}
Let $Q$ be a rooted tree quiver, so that its representation ring $R:=R(Q)$ has a finite decomposition
\[
R \cong \prod_{M \in L_Q} f_M R
\]
as a direct product of rings (Section \ref{fsuppsect}).  Then any indecomposable representation $V \not \simeq \charrep_M$, with $\fsupp V = M$, has nilpotent image in the factor $f_M R$.  Consequently,
each factor has a $\Z$-module decomposition
\[
f_M R \cong \Z f_M \oplus f_M \mathcal{N}
\]
where $\mathcal{N}$ is the nilradical of $R$.
\end{theorem}

The ``consequently'' part follows from the first statement because $f_M R$ is freely generated as a $\Z$-module by the indecomposable representations of $V$ with fine support $M$, and $f_M \charrep_M = f_M$.
Since $\Z$ is reduced (has no nilpotent elements), the rank function $r_M \colon R \to \Z$ restricts to the projection
\[
r_M \colon f_M R \to \Z f_M = \left( f_M R \right)_{red}
\]
so that, for any representation $V$, we have
\[
f_M V = r_M (V) f_M + n \qquad \qquad n \in \mathcal{N} . 
\]

\begin{theorem}\label{splittingthm}
Let $N \in \latt{Q}{\sigma}$ be maximal such that $\rank_N V \neq 0$.  Then for some $l > 0$, there exists a subrepresentation $U \simeq \charrep_N \otimes_K (\rank_N V^{\otimes l}) \subseteq V^{\otimes l}$ such that
\[
V^{\otimes l} \simeq U \oplus W
\]
for some $W \subseteq V^{\otimes l}$, necessarily with $\rank_N W = 0$.
\end{theorem}

Note that if the conclusion holds for some $l$, then it also holds for all $l' > l$.  We will refer to this theorem as the \textbf{splitting principle} for representations of rooted tree quivers.
Strictly speaking, we will only need to use that Theorem \ref{mainthm} implies the splitting principle in order to prove the theorems.  However, we will show that the two theorems are equivalent in order to expand our dictionary between $R(Q)$ and $\repq$.

\begin{proof}[Proof of equivalence of Theorems \ref{mainthm} and \ref{splittingthm}]
Suppose that Theorem \ref{mainthm} is true, and let $V$ be a representation of $Q$.  Then for sufficiently large $l$, in each factor $f_M R$ we have
\[
f_M V^l = 0 \iff r_M (V) = 0 . 
\]
For a fixed $l$ such that this holds, choose any maximal $N \in \latt{Q}{\sigma}$ such that $f_N V^l \neq 0$.  Then write $V^{\otimes l} \simeq \left( \bigoplus_i U_i \right) \oplus W$ where each $U_i$ is indecomposable with $r_N (U_i ) \neq 0$, and $r_N (W) = 0$.  Then for each $i$, $\fsupp U_i \geq N$ by (F5).
Now take any $N' \geq N$ maximal such that $V^{\otimes l}$ has a direct summand with fine support $N'$.  Properties (F3) and (F4) imply that $f_{N'} V^l \neq 0$, so $N' = N$ by maximality of $N$.
Thus $V^{\otimes l}$ has no direct summands with fine support strictly greater than $N$, so $\fsupp U_i=N$. By property (F6), each $U_i \simeq \charrep_N$.  This gives a decomposition as in Theorem \ref{splittingthm}.

Now assume that Theorem \ref{splittingthm} is true, and let $V$ be as in the hypotheses of Theorem \ref{mainthm}.
We proceed by induction on the order of
\[
\nonvan_V := \setst{M' \in \latt{Q}{\sigma} }{\rank_{M'} V \neq 0 } . 
\]
Since $V$ has fine support $M \in \latt{Q}{\sigma}$, we have that $V_\sigma \neq 0$, so $\# \nonvan_V \geq 1$.  Now choose any maximal element $N$ of $\nonvan_V$.
Applying Theorem \ref{splittingthm}, we get 
\[
V^{\otimes l} \simeq U \oplus {W}
\]
with $U \simeq \charrep_N \otimes_K (\rank_N V^{\otimes l})$.
%, and $\# \nonvan_{{W}} < \# \nonvan_V$ (in particular, $\rank_N W = 0$).  Now 
Property (F5) implies that $N \leq \fsupp V = M$.  Since we assumed $V \not \simeq \charrep_M$, property (F6) implies that $\rank_M V = 0$, so $N < M$.  Thus $f_M U = 0$
by equation (\ref{charreptimesfeqn}), so $f_M V^l = f_M {W}$.

By properties (F1) and (F2), every direct summand of $W$ has fine support less than or equal to $M$.  Then (F3) and (F4) imply that $f_M W = \sum_i f_M X_i$ where $X_i$ are the indecomposable summands of $W$ with fine support exactly $M$.  Since $\rank_M X_i \subseteq \rank_N X_i \subseteq \rank_N W = 0$, we know that no $X_i \not \simeq \charrep_M$.  Now by the induction hypothesis, each $f_M X_i$ is nilpotent, and hence $f_M W = f_M V^l$ is nilpotent.
\end{proof}

We illustrate the splitting principle with an example.

\begin{example}
Let $Q$ be the five subspace quiver labeled as below, and $\alpha$ a dimension vector for $Q$:
\[
Q = \fivesubspacemaps{1}{2}{3}{4}{5}{\sigma}{}{}{}{}{} \qquad \alpha = \fivesubspaced{2}{2}{2}{2}{2}{3} . 
\]
Let $V \in \repq$ be an indecomposable of dimension $\alpha$, so it can be thought of as a three dimensional vector space with a collection of five specified planes $\setst{V_i}{1 \leq i \leq 5}$.  Assume $V$ is general in the sense that the planes $V_i$ are in general position, with pairwise intersection of dimension one and the intersection of any three planes being zero.

In the notation of Example \ref{nsubspaceeg}, the element $J = \{1, 2 \} \in \latt{Q}{\sigma}$ is maximal such that the corresponding rank space, $\rank_J V = V_1 \cap V_2$, is nonzero.
So the splitting principal (Theorem \ref{splittingthm}) says that for some $l > 0$ we have $V^{\otimes l} \simeq U \oplus W$, where $U$ is an indecomposable direct summand such that
\[
U_\sigma = \left(\bigcap_{i \in J} V_i^{\otimes l}\right) \qquad \text{and} \qquad \dimv(U) \simeq \fivesubspaced{1}{1}{0}{0}{0}{1} . 
\]
Furthermore, we necessarily have that $W_1 \cap W_2 = 0$ since rank functors are additive and multiplicative.  Repeating this process with other pairs $J' = \{i, j \}$, for large $m$ we eventually get a decomposition
\[
V^{\otimes m} \simeq \bigoplus \charrep_{i, j} \oplus Z
\]
where the sum is taken over all two element subsets $\{i, j\} \subset \{1, 2, 3, 4, 5 \}$, and $\charrep_{i,j}$ is the analogous representation to $U$ above but whose support is the vertex set $\{i, j, \sigma \}$.
Furthermore, $Z$ must a representation of $Q$ given by a collection of subspaces with all pairwise intersections zero.
\end{example}

%%%%%%%%%%%%%%%%%%%%%%%%%%%%%%%%%%%%%%%%%%%%%%%
%				PROOF OF MAIN THEOREM							%
%%%%%%%%%%%%%%%%%%%%%%%%%%%%%%%%%%%%%%%%%%%%%%%
Theorem \ref{mainthm} will be proven by induction on the ``complexity'' of $Q$ (see below).  
When $M < \hat{1}_Q$, we can use the combinatorics of Section \ref{reducedrepssect} to utilize the induction hypothesis.
The case that $M = \hat{1}_Q$ is essentially computational, utilizing an inductive application of splitting principle.

\subsection{The case $M < \hat{1}_Q$}\label{caseonesect}
Using the connection between reduced representations of $Q$ and those of $Q_M$ given by Theorem \ref{charreppullbackthm}, we can easily handle this case by induction.  However, $Q_M$ may have more vertices than $Q$, so we must find another ordering to use induction on.  It turns out that the set of rooted tree quivers, $\rtrees$, can be well-ordered such that a reduced quiver $Q_M$ over $Q$ is less than or equal to $Q$, with equality only for $Q_M = Q$.  This \textbf{complexity ordering} turns out to be essentially lexicographical order, if we encode rooted trees in the right notation.

A sequence of rooted tree quivers $(\Gamma_1, \dotsc, \Gamma_n )$ defines a rooted tree quiver by extending each $\Gamma_i$ from its sink, then gluing these extensions together at their sinks:
\[
(\Gamma_1, \dotsc, \Gamma_n) := \vcenter{
\xymatrix@R=-1ex{
{\Gamma}_1 \ar[dr]		& \\
{\vdots}		& {\sigma} \\
{\Gamma}_n \ar[ur]		&	}} .
\]
Any rooted tree arises in this way.  We recursively define a well-ordering on $\rtrees$.
Let $\rtrees_i \subset \rtrees$ be the collection of rooted tree quivers which have a path of length $i$, and no longer path.  Then $\rtrees_0$ has one element, the rooted tree with one vertex, and for $k \geq 1$,
\[
\rtrees_k = \setst{(\Gamma_1,\dotsc, \Gamma_n)}{\Gamma_1 \in \rtrees_{k-1} \text{ and } \Gamma_i \in \bigcup_{j=0}^{k-1} \rtrees_j} . 
\]
The set $\{\rtrees_k \}$ is a partition of $\rtrees$.  Now let $\Gamma, \Lambda$ be arbitrary rooted tree quivers, say with $\Gamma \in \rtrees_k$ and $\Lambda \in \rtrees_l$.   If $k < l$, then define $\Gamma < \Lambda$.  If $k =l$, then we can write
\begin{equation}\label{treeseqeqn}
\Gamma = (\Gamma_1,\dotsc, \Gamma_n),%{ \Gamma_i \geq \Gamma_{i+1}} 
\qquad \Lambda = (\Lambda_1, \dotsc, \Lambda_m)%{\Lambda_i \geq \Lambda_{i+1}}
\end{equation}
with $\{ \Gamma_i \}$ and $\{ \Lambda_i \}$ contained in $\bigcup_{j=0}^{k-1} \rtrees_j$, which we can assume is well-ordered.
By requiring that the sequences in (\ref{treeseqeqn}) be weakly decreasing, these expressions are unique.  To simplify the order condition below, we will assume $n=m$ by filling out the shorter sequence with symbols $\emptyset$ which we take to be less than all rooted trees.
% (up to terms of $\emptyset$).
Now define $\Gamma < \Lambda$ if and only if $\Gamma_i < \Lambda_i$ for the smallest $i$ such that $\Gamma_i \neq \Lambda_i$.  In other words, the ordering in $\rtrees_k$ is lexicographical with respect to the ordering on $\bigcup_{j=0}^{k-1} \rtrees_j$.

\begin{prop}
If $Q' \xto{c} Q$ is a reduced quiver over $Q$, then $Q' \leq Q$ in the complexity ordering, with equality if and only if $c= id$.
\end{prop}
\begin{proof}
Let $Q' = (Q'_1, \dotsc, Q'_m)$ and $Q = (Q_1, \dotsc, Q_n)$ in the notation above.  Then the structure map $c$ sends each $Q'_i$ into some $Q_{\varphi(i)}$, defining a function
\[
\varphi \colon [1, \dotsc, m] \to [1, \dotsc, n] \qquad c (Q'_i) \subseteq Q_{\varphi(i)} . 
\]
Each $Q'_i$ is a reduced quiver over $Q_{\varphi(i)}$, by the construction of \S\ref{reducedquiversect}.  We prove the proposition by induction on the number of vertices of $Q$.  By the induction hypothesis and our notational convention, $Q'_1 \leq Q_{\varphi(1)} \leq Q_1$.  If any of these inequalities is strict, then $Q' < Q$ and we are done. If these are equalities, then by induction $c$ restricts to the identity on $Q'_1$, and the reducedness assumption for $Q'$ implies that no other $Q'_i$ maps into $Q_{\varphi(1)}$. Hence $(Q'_2, \dotsc, Q'_m)$ is a reduced quiver over $(Q_2, \dotsc, Q_n)$, with structure map the restriction of $c$.  By the induction hypothesis, $(Q'_2, \dotsc, Q'_m) \leq (Q_2, \dotsc, Q_n)$, with equality if and only if $c$ restricts to the identity on $(Q'_2, \dotsc, Q'_m)$. Hence $Q' \leq Q$ with equality if and only if $c$ is the identity on all of $Q'$.
\end{proof}

\begin{prop}\label{pushpullrepringprop}
For $M \in \latt{Q}{\sigma}$ with $M < \hat{1}_Q$, let $c := c_M$.  Then the composition 
\[
c_* \circ c^* \colon R \to R(Q_M) \to R
\]
is the identity on $f_M R$.
\end{prop}
\begin{proof}
Let $V \in \repq$ with $\fsupp V = M$.  Since both $c_*$ and $c^*$ are additive, so is the composition.  For any $N$, (F2) implies that
\[
\fsupp ( \charrep_N \otimes V) \leq M
\]
and so $c_* c^* (\overline{\charrep_N V}) = \overline{\charrep_M} (\overline{\charrep_N V}) = \overline{\charrep_N V}$ by Lemma \ref{pullpushtensorlemma} and Corollary \ref{charreptensorcor}.  Since $c_*$ and $c^*$ are both additive, we can use the definition of $f_M$ to compute
\begin{align*}
c_* c^* (f_M V) &= c_* c^* \left( \sum_{N \leq M} \mu(N, M) \overline{\charrep_N V} \right) = \sum_{N \leq M} \mu(N, M) c_* c^* \left( \overline{\charrep_N V} \right) \\
&= \sum_{N \leq M} \mu(N, M) \overline{\charrep_N V} = f_M V . 
\end{align*}
By property (F4), $f_M R$ is generated as a $\Z$-module by $\setst{f_M V}{\fsupp V = M}$; hence $c_* c^*$ is the identity on $f_M R$.
\end{proof}

This shows that $c^*$ gives an embedding of the factor $f_M R$ into the representation ring of $Q_M$.  Now we'll see that the image lies in the ``top'' factor of $R(Q_M)$.

\begin{prop}\label{fpullbackprop}
For any $M \in \latt{Q}{\sigma}$, $c_M^* f_M = f_{\hat{1}}$ in $R(Q_M)$.
\end{prop}
\begin{proof}
Let $A$ be any finite lattice, and $\delta := \sum_{x \in A} \mu(x, \hat{1}) x$ in the M\"obius algebra of $A$.  There is a factorization
\[
\delta = \prod_{x \prec \hat{1}} (\hat{1} -x)
\]
(cf. \cite[Cor.~3.9.4]{stanleyenumcombin}).  By Theorem \ref{charreppullbackthm}, we know that $c_M^* (\overline{\charrep_N}) = \overline{\charrep_{\pi^* (N)}}$ (in $R_\sigma$).  Lemma \ref{coatomlemma} gives that $\pi^*(M) = \hat{1}$, and $N \prec M$ if and only if $\pi^* (N) \prec \hat{1}$.  Since $c_M^*$ is a ring homomorphism,
\[
c_M^* (f_M) = \prod_{N \prec M} \left( c_M^* (\overline{\charrep_M}) - c_M^* (\overline{\charrep_N}) \right) = \prod_{ A \prec \hat{1}} \left( \overline{\charrep_{\hat{1}}} - \overline{\charrep_{A}} \right) = f_{\hat{1}} . \qedhere
\]
\end{proof}

\begin{proof}[Proof of Theorem \ref{mainthm} for $M < \hat{1}$]
Assume that Theorem \ref{mainthm} holds for rooted tree quivers which are less complex than $Q$.
The last two propositions and the induction hypothesis give an injective ring homomorphism
\[
c^* \colon f_M R \into f_{\hat{1}} R(Q_M) . 
\]
If $V \not \simeq \charrep_M$ is indecomposable with fine support $M$, then $\rank_{Q_M} (c^* V) = \rank_M V = 0$, so $c^* V$ has no direct summands of $\charrep_{\hat{1}} = \id_{Q_M}$. Thus
$c^*(f_M V) = f_{\hat{1}} c^* V$ is nilpotent by the induction hypothesis, and so $f_M V$ is also nilpotent since $c^*$ is injective.
\end{proof}
%%%%%%%%%%			 M = 1		%%%%%%%%%%%%%%%%%%%%%
\subsection{The case $M = \hat{1}_Q$}
This case is essentially computational, and will require a number of technical lemmas.
We will need a few facts from linear algebra.  
\comment{
The first is elementary, but we would like to have it stated for the record.
\begin{lemma}\label{sumintersectlemma}
Let $A, B, C$ be subspaces of some (finite dimensional) vector space $V$.  If either $C \subseteq A$ or $C \supseteq B$, then
\[
C \cap (A \oplus B) = (C \cap A) \oplus (C \cap B) \quad \square
\]
\end{lemma}
}
The following two lemmas roughly say that subspaces of a vector space become ``more spread out'' as we take tensor powers.

\begin{lemma}
Let $V$ and $W$ be finite dimensional vector spaces.
If $A_1, A_2, X \subset V$ and $B_1, B_2, Y \subset W$ are subspaces such that $X \cap A_1 = Y \cap B_2 = 0$, then
\[
\Bigl( X \otimes Y \Bigr) \cap \Bigl(A_1 \otimes B_1 + A_2 \otimes B_2\Bigr)  = 0 . 
\]
\end{lemma}
\begin{proof}
Choose projections $\pi_X \colon V \to X$ and $\pi_Y \colon W \to Y$ such that $A_1 \subseteq \ker \pi_X$ and $B_2 \subseteq \ker \pi_Y$.  Then
\[
(\pi_X \otimes \pi_Y)(A_1 \otimes B_1 + A_2 \otimes B_2) \subseteq (\pi_X \otimes \pi_Y)(A_1 \otimes B_1) + (\pi_X \otimes \pi_Y)( A_2 \otimes B_2) = 0
\]
hence the intersection is empty.%  Alternatively, choose complements for $X$ and $Y$ containing $A_1$ and $B_2$, and expand the direct sum expression of $V \otimes W$ using these complements.  The two intersectands clearly lie in different summands.
\end{proof}

\begin{lemma}\label{subspacepowerlemma}
Let $\{ V_i \}_{i =1}^n$ and $W$ be a collection of subspaces of a finite dimensional vector space $V$, and suppose that $W \cap V_i = 0$ for all $i$.  Then $W^{\otimes s} \cap \sum V_i^{\otimes s} = 0 $ for $s \geq n$.
\end{lemma}

\begin{proof} First note that it is enough to show that the intersection is 0 for $s=n$, since for $s>n$, 
\[
W^{\otimes s} \cap \sum V_i^{\otimes s} \subseteq \left( W^{\otimes n}\otimes V^{\otimes s-n} \right) \cap \left( \sum V_i^{\otimes n} \otimes V^{\otimes s-n} \right) \subseteq \left( W^{\otimes n} \cap \sum V_i^{\otimes n} \right) \otimes V^{\otimes s-n} = 0 . 
\]
Use induction on $n$, the base case being $n = 2$. 
For this, take $X = Y = W$, $A_i = V_1$, and $B_i = V_2$ in the above lemma.
For the induction step, take another subspace $V_{n+1}$ such that $W \cap V_{n+1} =0$, and assume that 
\[
W^{\otimes s} \cap \sum_{i=1}^n V_i^{\otimes s} = 0 
\]
for $s \geq n$.  Now using the previous lemma again, with 
\[
X = W^{\otimes s}, \quad Y = W, \quad A_1 = \sum_{i=1}^{n} V_i^{\otimes s} ,\quad B_1 = V,\quad A_2 = V_{n+1}^{\otimes s},\quad B_2 = V_{n+1}
\]
we find that
\[
W^{\otimes s+1} \cap \sum_{i=1}^{n+1} V_i^{\otimes s+1} \subseteq W^{\otimes s} \otimes W \cap \left( \left( \sum_{i=1}^{n} V_i^{\otimes s} \right) \otimes V + V_{n+1}^{\otimes s} \otimes V_{n+1} \right) = 0 .  \qedhere
\]
\end{proof}

\begin{lemma}
Let $M, N \in \latt{Q}{\sigma}$ and $V \in \repq$ with $\fsupp V \leq N$.  Choose $u_M, u_N$ some fixed nonzero vectors in $(\charrep_M)_\sigma$ and $(\charrep_N)_\sigma$, respectively.
Then there exists a morphism
\[
\theta \colon \charrep_M \otimes V \to \charrep_N \otimes V
\]
such that $\theta_\sigma (u_M \otimes v) = u_N \otimes v$ for all $v \in V_\sigma$.
\end{lemma}
\begin{proof}
The map $\theta$ will be the following composition, to be explained one step at a time:
\[
\theta \colon \charrep_M \otimes V \xto{id \otimes f} \charrep_M \otimes \charrep_N \otimes V \xto{g \otimes id} \charrep_{M \wedge N} \otimes V \xto{\rho \otimes id} \charrep_N \otimes V . 
\]

Since $\fsupp V \leq N$, there is a decomposition
$\charrep_N \otimes V \simeq V \oplus W$
for some $W$, giving a map $f \colon V \into \charrep_N \otimes V$ which sends $v \mapsto u_N \otimes v$ at $\sigma$.  Tensoring this with $\charrep_M \xto{id} \charrep_M$ gives the first map.
From Corollary \ref{charreptensorcor}, there is a map $g \colon \charrep_M \otimes \charrep_N \to \charrep_{M \wedge N}$ which is an isomorphism at $\sigma$.  Tensoring this with the identity on $V$ gives the second map.
By Theorem \ref{charrephomthm}, we can choose a map $\rho \colon \charrep_{M \wedge N} \to \charrep_N$ such that $\rho_\sigma \left( g(u_M \otimes u_N) \right)= u_N$.  Tensoring this with the identity on $V$ gives the third map, so that the composition maps $u_M \otimes v$ to $u_N \otimes v$ at $\sigma$.
\end{proof}

\begin{lemma}\label{mainthmlemmatwo}
Let $\{ S_1, \dotsc, S_n \} \subseteq \latt{Q}{\sigma}$ be an arbitrary subset, and $V \in \repq$ such that $\fsupp V \leq T \in \latt{Q}{\sigma}$.  Let $u_i \in (\charrep_{S_i})_\sigma$ be nonzero, and $z \in \charrep_T$ also nonzero.  Define
\[
Y := \left( \charrep_T \otimes V \right) \oplus \bigl( \bigoplus_i \charrep_{S_i} \otimes V \bigr)
\]
so that $Y_\sigma = (Kz \otimes V_\sigma) \oplus (\bigoplus_i Ku_i \otimes V_\sigma)$, and set $w_i := u_i - z$.

Then Y has a direct sum decomposition $Y \simeq (\charrep_T \otimes V) \oplus W$, where $W \subset Y$ is a subrepresentation such that
\[
W_\sigma = \bigoplus_i Kw_i \otimes V_\sigma . 
\]
\end{lemma}
\begin{proof}
For each $i$, the previous lemma gives a morphism
\[
\theta_i \colon \charrep_{S_i} \otimes V \to \charrep_T \otimes V
\]
such that $(\theta_i)_\sigma (u_i \otimes v) = z \otimes v$ for all $v \in V_\sigma$.  Writing $X := \bigoplus_i \charrep_{S_i} \otimes V$, we get a map
\[
\phi := (id_X - \sum_i \theta_i ) \colon X \to Y . 
\]
Since the image of $\sum_i \theta_i$ is contained in $\charrep_T \otimes V$, the map $\phi$ is injective and $\im \phi \cap \charrep_T \otimes V=0$.  Hence $W:= \im \phi$ is a complementary subrepresentation to $\charrep_T \otimes V$, and $\phi_\sigma ( u_i \otimes v) = (u_i - z) \otimes v = w_i \otimes v$ for $v \in V_\sigma$.
\end{proof}

To make the induction step, we will actually need a stronger version of the splitting principal. We will show, however, that this stronger version follows as a corollary of Theorem \ref{splittingthm}, so that we can apply it under the induction hypothesis.

\begin{corollary}[to Theorem \ref{splittingthm}]\label{supersplittingcor}
Let $Z \subseteq V_\sigma$ be a subspace such that both $Z \subseteq \rank_M V$ for some $M \in \latt{Q}{\sigma}$, and $Z \cap \rank_N V = 0$ for $N \nleq M$.  Then for some $l \geq 0$, $V^{\otimes l}$ has a subrepresentation $U  \simeq \charrep_M \otimes_K Z^{\otimes l}$ such that
\[
V^{\otimes l} \simeq U \oplus W
\]
for some $W \subseteq V^{\otimes l}$.
\end{corollary}
\begin{proof}
Assume that we've proven Theorem \ref{splittingthm} for a quiver $Q$.  Note that the corollary holds as stated if and only if it holds after replacing $Z$ and $V$ with $Z^{\otimes k}$ and $V^{\otimes k}$ in the hypotheses, for any $k > 0$.  Similarly, a power $V^{\otimes k}$ can be replaced with a direct summand $W \subset V^{\otimes k}$ containing $Z^{\otimes k}$ at any point in the proof, since direct summands of $W^{\otimes l}$ are also direct summands of $V^{\otimes kl}$.

First we reduce to the case that $\rank_N V \neq 0$ only for $N$ which are comparable to $M$.  To this end, suppose that there exists some $N \ngeq M$ such that $\rank_N V \neq 0$, and take a maximal $N$ with this property.  By the theorem, we get
\[
V^{\otimes l} \simeq U \oplus W
\]
with $U \simeq \charrep_N \otimes_K (\rank_N V^{\otimes l})$.  Hence,
\[
Z^{\otimes l} \subseteq \rank_M V^{\otimes l} = \rank_M U \oplus \rank_M W = \rank_M W
\]
where $\rank_M U = 0$ since $M \nleq N$.  Then we can work in $W$, with $\rank_N W = 0$.  Repeating this process, we can assume all nonzero rank spaces are comparable to $M$.

Now by Lemma \ref{subspacepowerlemma}, we can replace $V$ and $Z$ with some tensor powers of themselves and assume that $Z \cap X = 0$, where
\[
X := \left( \sum_{N \nleq M} \rank_N V \right) . 
\]
Choose a vector space projection $\pi \colon V_\sigma \onto Z$ such that $X \subseteq \ker \pi$ and $\pi |_Z = id_Z$.

Now we claim that for large enough $l$, $V^{\otimes l} \simeq A \oplus B$ with $A_\sigma \subseteq \ker(\pi^{\otimes l})$ and $M$ maximal such that $\rank_M B \neq 0$:  suppose for contradiction that this is not possible, and take a decomposition as above such that $A_\sigma \subseteq \ker(\pi^{\otimes l})$ and the quantity
\begin{equation}\label{splitcormineqn}
\# \setst{M' \nleq M}{\rank_{M'} B \neq 0}
\end{equation}
is minimal.  If this quantity is 0, then the claim is verified.  If not, let $N \nleq M$ be maximal such that $\rank_N B \neq 0$, and apply the theorem to get $B^{\otimes l'} \simeq C \oplus D$ with $C_\sigma = \rank_N B^{\otimes l'}$ and $\rank_N D = 0$.  Then
\begin{equation}\label{splitcortensoreqn}
\left(V^{\otimes l} \right)^{\otimes l'} \simeq A^{\otimes l'} \oplus \cdots \oplus B^{\otimes l'} \simeq A^{\otimes l'} \oplus \cdots \oplus C \oplus D
\end{equation}
and $\setst{M' \nleq M}{\rank_{M'} D \neq 0} \subsetneq \setst{M' \nleq M}{\rank_{M'} B \neq 0}$.  All the summands represented by $\cdots$ have at least one tensor factor of $A$, so each is also in $\ker ( \pi^{\otimes l l'} )$.   Also $C_\sigma \subseteq \rank_N V^{\otimes l l'} \subseteq \ker (\pi^{\otimes ll'})$, since $N \nleq M$, and thus
\[
\pi^{\otimes l l'} \left( A^{\otimes l'}_\sigma \oplus \cdots \oplus C_\sigma \right) \subseteq \pi^{\otimes l l'} (A^{\otimes l'}_\sigma) + \pi^{\otimes l l'} (\cdots) + \pi^{\otimes l l'} (C_\sigma) = 0
\]
and so (\ref{splitcortensoreqn}) contradicts the minimality of the quantity in (\ref{splitcormineqn}).

Now, again replacing $V^{\otimes l}$ with $V$, and $\pi^{\otimes l}$ with $\pi$, we have a decomposition $V \simeq A \oplus B$ with $A_\sigma \subseteq \ker \pi$ and $M$ maximal such that $\rank_M B \neq 0$.  We can apply the theorem to $B$ to get
\[
V^{\otimes l} \simeq A^{\otimes l} \oplus \cdots B^{\otimes l} \simeq A^{\otimes l} \oplus \cdots \oplus C \oplus D = A' \oplus C \oplus D
\]
where the last equality is just collecting summands.  Here we have that $A'_\sigma \subseteq \ker(\pi^{\otimes l})$, just as in the argument of the last paragraph, and $C \simeq \charrep_M \otimes_K C_\sigma$.  But also $\rank_M D = 0$, so
\[
Z^{\otimes l} \subseteq \rank_M V^{\otimes l} = \rank_M A' \oplus \rank_M C \oplus \rank_M D = \rank_M A' \oplus \rank_M C
\]
and so we can work in $A' \oplus C$.  Replacing $A' \oplus C$ with $V$ and $Z^{\otimes l}$ with $Z$, we can now assume that $V$ has a decomposition $V \simeq C \oplus A'$ with $C \simeq \charrep_M \otimes_K C_\sigma$ and $A'_\sigma \subseteq \ker \pi$.  The last containment implies that
\[
\ker \pi = \ker \pi \cap (C_\sigma \oplus A'_\sigma) = (\ker \pi \cap C_\sigma) \oplus A'_\sigma . 
\]
Via a linear combination of scalar endomorphisms, we can write $C \simeq C' \oplus C''$, with $\dim_K C'_\sigma = \dim_K Z$ and $C''_\sigma = \ker \pi \cap C_\sigma$, then let $U: =C'$ and $W:=A' \oplus C''$.  Now $W_\sigma = \ker \pi$, by dimension count, and since $\pi |_Z = id_Z$, it must be that $Z \cap W_\sigma = 0$.  By Lemma \ref{basechangelemma}, we can make a change of basis in $V$ to get  $U_\sigma = Z$.   This proves the corollary.
\end{proof}

\begin{proof}[Proof of Theorem \ref{mainthm} for M = $\hat{1}_Q$]  We are assuming that Theorem \ref{splittingthm} holds for quivers which are less complex than $Q$.  Suppose $V$ is indecomposable and $\fsupp V = \hat{1}_Q$, but $V \not \simeq \id_Q$.  Then $\rank_Q V = 0$ by Theorem \ref{kinserrankthm}, and we need to show that $f_{\hat{1}_Q} V$ is nilpotent.  We are assuming that the theorem holds for rooted tree quivers which are less complex than $Q$; in particular, it holds for subquivers of $Q$, so we can use the usual extension and gluing cases.

\extcase
If $\rank_P V|_P \neq 0$, then by Theorem \ref{kinserrankthm} we have a direct sum decomposition
\[
V|_P \simeq (\id_P \otimes_K \rank_P V) \oplus W
\]
with $\rank_P W = 0$.  But because $\rank_Q V = 0$, Lemma \ref{glueextranklemma} implies that $\rank_P V|_P \subseteq \ker V_\alpha$.  Then $W$ would extends to a direct summand of $V$ over $Q$ by setting $W_\sigma = V_\sigma$, contradicting the indecomposability of $V$.  Hence $\rank_P V = 0$.

Then by induction, $f_{\hat{1}_P} V \in R(P) \subset R(Q)$ is nilpotent, so $f_{\hat{1}_P} V^l = 0$ for large $l$. This implies that there is a direct sum decomposition of the restriction
\[
V|_P^{\otimes l} \simeq \bigoplus_i V_i \oplus Z
\]
such that each $V_i$ is indecomposable with $(V_i)_\tau \neq 0$ and $\fsupp V_i < \hat{1}_P$, and $Z_\tau = 0$.
The lattice $\latt{Q}{\sigma} = J (\latt{P}{\tau})$ has a unique coatom 
\[
C = \latt{P}{\tau} \setminus \{ \hat{1}_P \} = \gen{C_1, \dotsc, C_k }
\]
where $\{C_i\}$ is the set of coatoms of $\latt{P}{\tau}$.
We will show that for such an $l$ as above, $\fsupp V^{\otimes l} \leq C$.
%$\charrep_C \otimes V^{\otimes l} \simeq V^{\otimes l} \oplus U$ with $\sigma \not \in \supp U$.

First, restricting to $P$ we get:
\[
(\charrep_C \otimes V^{\otimes l})|_P = (\charrep_C )|_P \otimes (V^{\otimes l})|_P = \bigl(\bigoplus_j \charrep_{C_j} \bigr) \otimes \bigl(\bigoplus_i V_i \oplus Z \bigr) = \bigoplus_i Y_i \oplus \tilde{Z}
\]
where $Y_i := \bigl( \bigoplus_j \charrep_{C_j} \bigr) \otimes V_i$, and $\tilde{Z} = \bigl( \bigoplus_j \charrep_{C_j} \bigr) \otimes Z$.
For each indecomposable summand $V_i$, choose a coatom $T_i \in \{C_1, \dotsc, C_k \}$ such that $T_i \geq \fsupp V_i$.  At least one such $T_i$ exists because $\fsupp V_i < \hat{1}$.  Let $u_j \in (\charrep_{C_j})_\sigma$ be nonzero such that $(\charrep_C)_\alpha (u_j) = u_\sigma \in (\charrep_C)_\sigma$ for all $j$ (that is, the $u_j$ and $u_\sigma$ are part of a standard basis for $\charrep_C$).
Then apply Lemma \ref{mainthmlemmatwo} with $\{S_1, \dotsc, S_n \} = \{C_1, \dotsc, C_k \} \setminus \{ T_i \}$ to write each
\[
Y_i  = (\charrep_{T_i} \otimes V_i) \oplus \bigl( \bigoplus_{C_j \neq T_i} \charrep_{C_j} \otimes V_i \bigr) \simeq (\charrep_{T_i} \otimes V_i) \oplus W_i
\]
with $(W_i)_\tau = \bigoplus_j K(u_j - u_i) \otimes (V_i)_\tau \subseteq \ker (\charrep_C \otimes V^{\otimes l})_\alpha$.  Then since $\fsupp V_i \leq T_i$, each $\charrep_{T_i} \otimes V_i \simeq V_i \oplus U_i$ with  $(U_i)_\tau = 0$, and by setting $X := \bigl( \bigoplus_i W_i \oplus U_i \bigr) \oplus \tilde{Z}$ we get
\[
(\charrep_C \otimes V^{\otimes l})|_P \simeq (\bigoplus_i V_i) \oplus X \simeq V^{\otimes l}|_P \oplus X
\]
with $X_\sigma \subseteq \ker(\charrep_C \otimes V^{\otimes l})_\alpha$.  Since this isomorphism sends $u_i \otimes v$ to $(v, 0)$ for $v \in V^{\otimes l}_\tau$, the map $(\charrep_C \otimes V^{\otimes l})_\alpha$ acts on the right hand side just as $V^{\otimes l}_\alpha$, so the decomposition extends to give $\charrep_C \otimes V^{\otimes l} \simeq V^{\otimes l} \oplus X$ with $\sigma \not \in \supp X$.  Hence $\fsupp V^{\otimes l} \leq C$, and so $f_{\hat{1}_Q} V^l = 0$.

\gluecase Suppose $Q$ is a gluing of $P$ and $S$, as usual, but also take $P$ to be a one point extension of some smaller quiver (so there is a unique arrow $a \in P\arrows$ with $ha = \sigma$).  Denote by $\hat{1}_P$ and $\hat{1}_S$ the maximal elements of $\latt{P}{\sigma}$ and $\latt{S}{\sigma}$, respectively.  We show that $f_{\hat{1}_Q} V$ is nilpotent by induction on $d = \# \setst{M \in \latt{S}{\sigma}}{\rank_{(\hat{1}_P, M)} V \neq 0}$.  

If $d=0$, then $\rank_P V = \rank_{(\hat{1}_P , \hat{0}_S)} V = 0$, and the induction hypothesis gives (just as in the extension case)
%so for sufficiently large $l$, the induction hypothesis gives a decomposition $V|_P^{\otimes l} \simeq \bigoplus V_i$ with $\fsupp V_i < \hat{1}_P$.  
\[
\fsupp V|_P^{\otimes l} \leq C
\]
for large $l$, where $C$ is the unique coatom of $\latt{P}{\sigma}$.
But then, regardless of how $V|_S^{\otimes l}$ decomposes, we get that $\charrep_{(C, \hat{1}_S)} \otimes V^{\otimes l} = V^{\otimes l} \oplus U$, so $\fsupp V^{\otimes l} \leq (C, \hat{1}_S) < \hat{1}_Q$.  Hence $f_{\hat{1}_Q} V^l = 0$.

If $d > 0$, choose any $M \in \latt{S}{\sigma}$ maximal such that $Z := \rank_{(\hat{1}_P, M)} V \neq 0$.  Then $M \neq \hat{1}_S$ since $\rank_{(\hat{1}_P, \hat{1}_S)} V = \rank_Q V = 0$ by hypothesis.  Now by maximality, $Z \subseteq \rank_M V$ and $Z \cap \rank_N V = 0$ for $N \nleq M$ in $\latt{S}{\sigma}$.   An inductive application of Corollary \ref{supersplittingcor} over $S$ then gives subrepresentations $U, W \subset V|_S^{\otimes l}$ such that $U \simeq \charrep_M \otimes_K Z^{\otimes l}$, and $V|_S^{\otimes l} = U \oplus W$.

But over $P$, since $Z^{\otimes l} \subseteq \rank_P V|_P^{\otimes l}$, Theorem \ref{kinserrankthm} gives subrepresentations $X, Y \subset V|_P^{\otimes l}$ such that $X = \id_P \otimes_K Z^{\otimes l}$ and $V|_P^{\otimes l} = X \oplus Y$-- in particular, $X_\sigma = Z^{\otimes l} = U_\sigma$.  By Lemma \ref{basechangelemma}, we can take $Y_\sigma = W_\sigma$.  So we have subrepresentations $A, B \subset V^{\otimes l}$ defined by
\[
A := \begin{cases}
X & \text{over }P \\
U & \text{over }S
\end{cases}
\qquad \qquad
B := \begin{cases}
Y & \text{over }P \\
W & \text{over }S
\end{cases}
\]
giving a direct sum decomposition $V^{\otimes l} \simeq A \oplus B$, since this holds over both $P$ and $S$.  Now since 
\[
A \simeq \charrep_{(\hat{1}_P, M)} \otimes_K (\rank_{(\hat{1}_P, M)} V^{\otimes l})
\]
we know that both $f_{\hat{1}_Q} A = 0$ and $\rank_{(\hat{1}_P, M)} B = 0$.  By the induction hypothesis in this gluing case, $f_{\hat{1}_Q} B$ is nilpotent, so $f_{\hat{1}_Q} (A + B) = f_{\hat{1}_Q} V^l$ is nilpotent.
\end{proof}

%%%%%%%%%%%%%%%%%%%%%%%%%%%%%%%%%%%%%%%%%%%%%%%%%%
%						COROLLARIES										%
%%%%%%%%%%%%%%%%%%%%%%%%%%%%%%%%%%%%%%%%%%%%%%%%%%
The completes the proof of the main result.
The commentary after the statement of Theorem \ref{mainthm} gives as an immediate corollary:
\begin{corollary}
The rank functions on $Q$ give an isomorphism
\[
R(Q)_{red} \xto{\sim} \prod_{M \in L_Q} \Z \qquad \quad \overline{V} \mapsto \bigl( r_M (V) \bigr) . 
\]
In particular, $R(Q)_{red}$ is a finitely generated $\Z$-module.
\end{corollary}

\begin{example}
Continuing Examples \ref{exampletwo} and \ref{latticeseg}, we can calculate that
\[
\rank_{\Z} R(Q)_{red} = \sum_{x \in Q_0} \#\latt{Q}{x} = 31
\]
where $\rank_{\Z}$ is the rank as an abelian group.  This is because each of the three source vertices contribute 1, and we can count $\# \latt{Q}{\tau} = 8$ and $\#\latt{Q}{\sigma} = 20$ from the diagrams in Example \ref{latticeseg}.
\end{example}

\begin{corollary}
For a fixed $V \in \rep Q$, only finitely many indecomposable representations appear as direct summands of the representations $\{ V^{\otimes i} \}_{i \geq 0}$.  In other words, there exists a finite set of indecomposables $\{V_k \}$ such that $\{ V^{\otimes i} \}_{i \geq 0}$ is contained in the subcategory additively generated by $\{V_k\}$.
\end{corollary}
\begin{proof}
The ring $R_{red}$ is a finitely generated $\Z$-module, hence integral over $\Z$.  This implies that $R$ is integral over $\Z$ also, and hence the subalgebra $\Z [V]$ is too.  Then $\Z[V]$ is a finitely generated $\Z$-module, which gives the conclusion.
\end{proof}

%%%%%%%%%%%%%%%%%%%%%%%%%%%%%%%%%%%%%%%%%%%%%%%%%%
%						CONCLUSION					%
%%%%%%%%%%%%%%%%%%%%%%%%%%%%%%%%%%%%%%%%%%%%%%%%%%
\section{Conclusion}
\label{conclusionsect}

If $Q$ is any quiver now, and $R(Q)_{red}$ is a finitely generated $\Z$-module, say that $Q$ is of \textbf{finite multiplicative type} (over $K$).  Since the ring $R(Q)$ generally depends on the field $K$, this property could also.
Then so far we know that Dynkin quivers of any orientation and rooted tree quivers are of finite multiplicative type over any field.
A natural question to ask is then, ``What other quivers are of finite multiplicative type?".
The first observation we make regarding this question is that if $Q$ is of finite multiplicative type over $K$, then so is any \textbf{minor} of $Q$; that is, any quiver obtained from $Q$ by contracting edges and removing any combination of vertices and edges (cf. \cite[\S1.7]{diestelgraphtheory}).

\begin{prop}
The set of finite multiplicative type quivers over a given field $K$ is minor closed.
\end{prop}
\begin{proof}
Let $Q$ be of finite multiplicative type.  If $Q'$ is obtained from $Q$ by contracting an edge, then there is an injective homomorphism of rings $R(Q') \into R(Q)$ given on representations by assigning the identity map to the contracted edge (cf. \cite[\S6]{herschende6}).
If $Q'$ is obtained from $Q$ by removing some vertex or edge, then $R(Q')$ is isomorphic to the quotient of $R(Q)$ by the ideal generated by $e_P$ for all connected subquivers $P$ containing that vertex or edge.
\end{proof}

%(is this really necessary?) The Clebsch-Gordan problem for the single loop quiver $\tilde{A}_0$ was studied as far back as the 1930's by Aitken \cite{aitkennormalform}, whose results were generalized and simplified by Littlewood \cite{littlewoodinducedcompound}.  In fact, Littlewood calculated the decomposition not only of tensor products, but of any Schur functor applied to a representation of $\tilde{A}_0$.  Some of these results have been independently rediscovered a number of times since then.  Most relevant to the study of quiver representations are papers by Martsinkovsky and Vlassov \cite{martsinkovskykxmods}, who also consider a representation ring with a different multiplicative structure, and \cite{herschendaffinean}, in which Herschend solves the Clebsch-Gordan problem for quivers of type $\tilde{A}_n$.

It is easy to see that the loop quiver $\tilde{A}_0$ is not of finite multiplicative type over an infinite field $K$.
%  Let $R' := \Z[K ; \cdot]$ be the semigroup ring of $K$, with respect to multiplication, with coefficients in $\Z$.  Then there is an embedding $R' \into R(\tilde{A}_0)$ sending $\lambda \in K$ to the scalar endomorphism of $K$ of trace $\lambda$.  Since $R'$ is reduced, but not finitely generated as a $\Z$-module when $K$ is infinite, $R(\tilde{A}_0)_{red}$ is not a finitely generated $\Z$-module either.
The trace of an endomorphism is additive with respect to direct sum, and multiplicative with respect to tensor product, hence extends to a ring homomorphism
\[
\trace \colon R(\tilde{A}_0) \onto K . 
\]
A field $K$ is infinite if and only if it is not a finitely generated $\Z$-module.  Since a field is reduced, the trace map factors through $R(\tilde{A}_0)_{red}$, so $\tilde{A}_0$ is not of finite multiplicative type when $K$ is infinite.
The case that $K$ is finite can be handled by a more sophisticated argument, provided by an anonymous referee.  The map sending an endomorphism to its characteristic polynomial can be used to map the Grothendieck ring of $\rep_K (\tilde{A}_0)$ into $W(K)$, the ring of universal Witt vectors over $K$ (see \cite[p.~330]{MR1878556} or \cite[Ch.~IX]{MR2284892} for Witt vectors, and \cite{MR523461} for the relation to K-theory).  That $W(K)$ is reduced follows easily from the definition of multiplication in $W(K)$ and that $K$ has no nilpotents. Since the Grothendieck ring is a quotient of $R(\tilde{A}_0)$, and its image in $W(K)$ is not finitely generated as a $\Z$-module, we get that $\tilde{A}_0$ is not of multiplicative finite type over a finite field either.

Since a graph is a tree if and only if it doesn't have a loop as a minor, the proposition above implies that a quiver of finite multiplicative type must be a tree.
The following example shows that not every tree is of finite multiplicative type.

\begin{example}
Let $Q$ be the quiver of type $\tilde{D}_4$, oriented to have ``crossing paths'':
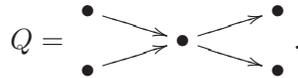
\[
Q = \crosspathsq . 
\]
Now consider the quiver
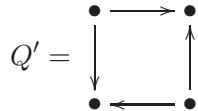
\[
Q' = \vcenter{
\xymatrix{
{\bullet}	\ar[r] \ar[d]	& {\bullet} \\
{\bullet}		& {\bullet} \ar[l] \ar[u] }}
\]
of type $\tilde{A}_3$.  There is a functor $f^* \colon \repq \to \rep (Q')$ given by
\[
\crosspathsmaps{V_1}{V_2}{V_3}{V_4}{V_5}{V_a}{V_b}{V_c}{V_d} \longmapsto \vcenter{
\xymatrix{
{V_1}	\ar[r]^{V_c V_a} \ar[d]^{V_d V_a}	& {V_2} \\
{V_5}		& {V_4} \ar[l]^{V_d V_c} \ar[u]_{V_c V_b} } }
\]
which, from the categorical viewpoint, is the composition of a representation $V \colon \mathscr{Q} \to \vscat$ with a certain functor $f \colon \mathscr{Q}' \to \mathscr{Q}$ (cf. \S\ref{rankfunctorsect}).  Equivalently, if we relax our definition of maps of directed graphs to allow arrows to be mapped to \emph{paths} in the target, then $f^*$ is still a pullback along a certain map $f \colon Q' \to Q$.
Applying the global tensor functor $\rkf_{Q'}$ to a representation of $Q'$ gives a representation in which all maps over the arrows of $Q'$ are isomorphisms (Theorem \ref{kinserrankthm}).  Such a representation induces a representation of $\tilde{A}_0$ by taking the underlying vector space to be the space at any vertex of $Q'$, and the endomorphism to be given by traversing around the cycle once, say clockwise.  This gives a functor $\loopfunctor \colon \rep^{\circ} (Q') \to \rep(\tilde{A}_0)$, where the domain is defined as the full subcategory of $\rep (Q')$ consisting of representations which have an isomorphism at each arrow of $Q'$.  So, summarily, we compose functors:
\[
\repq \xto{f^*} \rep (Q') \xto{\rkf_{Q'}} \rep^{\circ} (Q') \xto{\loopfunctor} \rep (\tilde{A}_0) . 
\]
Each of these functors respects direct sum and tensor product, and preserves $\id$, hence we have a ring homomorphism $R(Q) \to R(\tilde{A}_0)$.
%, which, composed with the trace homomorphism, gives a trace homomorphism
%\begin{equation}\label{tracemapeqn}
%\trace \colon R(Q) \to K .
%\end{equation}

%We want to see that this trace map is surjective.
%  Note that any other orientation of the graph $\tilde{D}_4$ is a rooted tree, so an analogously defined trace map cannot possibly be surjective for other orientations, by the main theorem.
For $\lambda \in K$ and $n \in \Z$, define a representation
\[
V_{\lambda} (n):= \crosspathsmaps{K^n}{K^n}{K^{2n}}{K^n}{K^n}{A}{B}{C}{D}
\]
with the maps given by the block form matrices
\[
A = {\twobyone{id}{0}} \qquad B={\twobyone{0}{id}}  \qquad C={\onebytwo{id}{id}} \qquad D= {\onebytwo{id}{J_\lambda(n)}},
\]
where $J_\lambda (n)$ is the Jordan block of size $n$ with eigenvalue $\lambda$.

Then the global tensor functor of $Q'$, applied to $f^* (V_\lambda (n))$, gives the representation 
\[
\rkf_{Q'} (f^* V_\lambda(n)) = \vcenter{
\xymatrix{
{K^n}	\ar[r]^{id} \ar[d]^{id}	& {K^n} \\
{K^n}		& {K^n} \ar[l]^{J_\lambda(n)} \ar[u]^{id} }}
\]
when $\lambda \neq 0$, and the 0 representation when $\lambda = 0$. Applying the functor $\loopfunctor$ gives the endomorphism $J_\lambda (n)$ of $K^n$ for $\lambda \neq 0$, and hence the image of the induced map $R(Q) \to R(\tilde{A}_0)$ contains the representations which have all eigenvalues nonzero.  The reduction of this image is not a finitely generated $\Z$-module, by a slight modification of the above argument for the loop quiver, and hence $R(Q)_{red}$ cannot be a finitely generated $\Z$ module either.
This shows that the tree quiver $Q$ is not of finite multiplicative type.
\end{example}

A similar argument with the pair of quivers
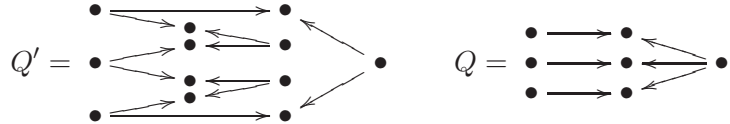
\[
Q' = \vcenter{\xymatrix@R=-1ex{
{\bullet} \ar[rr] \ar[dr] &	& {\bullet}	 \\
	 &	{\bullet}		 \\
& {\bullet} & {\bullet} \ar[ul] \ar[l] 	\\
{\bullet} \ar[ur] \ar[dr] &		&	&	{\bullet} \ar[uuul] \ar[dddl] \\
& {\bullet} & {\bullet} \ar[dl] \ar[l] 	\\
	 &	{\bullet}	 \\
{\bullet} \ar[rr] \ar[ur] &	& {\bullet}	}}
\qquad
Q = \extEsixa{\bullet}{\bullet}{\bullet}{\bullet}{\bullet}{\bullet}{\bullet}
\]
shows that this $Q$ is also not of finite multiplicative type.  To approach the classification of all quivers of finite multiplicative type, we suggest an analogy to the classification of quivers of finite representation type (cf. \cite{gabriel}).  It is well-known that a quiver $Q$ is of finite representation type if and only if $Q$ is a Dynkin diagram, of any orientation.  This is equivalent to saying that $Q$ does not have a minor of type $\tilde{A}_0, \tilde{D}_4, \tilde{E}_6, \tilde{E}_7$, or $\tilde{E}_8$ (of any orientation).  There are only finitely many orientations of a given graph, so the finite set of quivers of these five types gives a finite set of ``obstructions'' to a quiver being of finite representation type; these are sometimes called \textbf{forbidden minors}.

It is an open problem to give forbidden minors that characterize quivers of finite multiplicative type (over $\C$ even).  The Tree Theorem of J.B. Kruskal (cf. \cite{kruskaltreetheorem}) can be applied to show that the set of tree quivers is well-quasi-ordered; with this, we can at least say that the finite multiplicative type property can be characterized over a given field $K$ by a \emph{finite} set of forbidden minors.
%However, since this property depends on the orientation of $Q$ (and possibly the field $K$), it might not be as  describable as in the case of finite representation type.

\bibliographystyle{alpha}
\bibliography{ryanbiblio}

\end{document}